\documentclass[12pt]{amsart}

 \usepackage[T1]{fontenc}
\usepackage{amsmath,amssymb,amsthm}
\usepackage{mathtools}
\usepackage{graphicx}
\usepackage{geometry}
\usepackage{caption}
\usepackage{booktabs}
\usepackage{subcaption}
\usepackage{listings}
\usepackage{graphicx}
\usepackage{verbatim}
\usepackage{harvard}
\usepackage{hyperref}
\usepackage[main=english, polutonikogreek]{babel}
\newcommand{\pvn}{\par\vspace{.7ex}\noindent}
\newcommand{\pv}{\par\vspace{.7ex}}

\newcommand{\N}{{\mathbb{N}}}

\newcommand{\R}{{\mathbb{R}}}

\newcommand{\F}{{\mathbb{F}}}

\newtheorem{theorem}{Theorem}[section]

\theoremstyle{definition}
\newtheorem{definition}[theorem]{Definition}

\theoremstyle{remark}

\newcommand\tab[1][1cm]{\hspace*{#1}}
\title{Bridging Classical and Modern Approaches to Thales' Theorem}
\author{Piotr Błaszczyk, Anna Petiurenko}

\begin{document}
\maketitle
	\begin{abstract}
In this paper, we reconstruct Euclid's theory of similar triangles, as developed in Book VI of the \textit{Elements}, along with its 20th-century counterparts, formulated within the systems of Hilbert, Birkhoff, Borsuk and Szmielew, Millman and Parker, as well as Hartshorne. In the final sections, we present recent developments concerning non-Archimedean fields and mechanized proofs.

Thales' theorem (VI.2) serves as the reference point in our comparisons. It forms the basis of Euclid's system and follows from VI.1—the only proposition within the theory of similar triangles that explicitly applies the definition of proportion.

 Instead of the ancient proportion, modern systems adopt the arithmetic of line segments or real numbers. Accordingly, they adopt other propositions from Euclid's Book VI, such as VI.4, VI.6, or VI.9, as a basis.

In §\,10, we present a system that, while meeting modern criteria of rigor, reconstructs Euclid's theory and mimics its deductive structure, beginning with VI.1. This system extends to automated proofs of Euclid's propositions from Book VI.

Systems relying on real numbers provide the foundation for trigonometry as applied in modern mathematics. In §\,9, we prove Thales' theorem in geometry over the hyperreal numbers. Just as Hilbert managed to prove Thales' theorem without referencing the Archimedean axiom, so do we by applying the arithmetic of the non-Archimedean field of hyperreal numbers.
	\end{abstract}
	

 \tableofcontents
\section{Interpretating Thales' Theorem and Euclidean Proportion}
1. Thales' theorem, also known as the intercept theorem or the fundamental theorem of proportionality, plays a central role in Euclid's theory of similar figures, as developed in Book VI of the \textit{Elements} \cite{fitz2008}. The first proposition of this book, VI.1, states that triangles with the same height are to each other as their bases. 
Proposition VI.2 is what has been referred to as Thales' theorem since the 19th century. Propositions VI.4–7 outline the criteria for similar triangles, with VI.4 stating that in equiangular triangles,  corresponding sides are proportional.

In modern reconstructions of Euclid's theory of similar figures, Proposition VI.4 is crucial because it forms the foundation of trigonometry. However, for the reasons explained below, modern geometry has abandoned the ancient concept of proportion. Therefore, the general objective is to reestablish VI.4 on grounds independent of Euclidean proportion as defined in Book V of the \textit{Elements}.

\pvn 2. Euclidean proportion is the most significant ancient Greek theory transmitted to early modern mathematics. In contrast to Euclidean rigor, early modern mathematicians applied it in unorthodox ways. Although governed by the Archimedean axiom, the theory was applied to both standard and infinitesimal triangles.

It laid the foundations for early modern optics, mechanics, and the advancements of 17th-century calculus. Viewed from the perspective of mathematical techniques, Newton's \textit{Principia} represents a synthesis of Euclidean proportion and infinitesimals.

In contemporary mathematics, trigonometry encodes the Euclidean theory of similar triangles. Contrary to the widespread view advanced by 20th-century structuralism in the philosophy of mathematics, trigonometry is a part of modern calculus that does not derive solely from the axioms of real numbers. 

\pvn 3. Euclidean proportion and calculus converge in determining 
 the derivative of $\sin x$. To this end, the following trigonometric identity is essential:
\[\sin(x+h)-\sin x= 2\sin\frac{h}2\cos (x+\frac{h}2), \] 
which rests on Euclidean principles of similar triangles.

Furthermore, evaluating the limit 
\[\lim\limits_{h\rightarrow 0}\frac{\sin h}{h}\] 
also necessitates reference to Euclidean geometry.

Moreover, trigonometry, as applied in calculus, requires the radian measure of angles, which, in turn, relies on the final proposition of Euclid's Book VI. 

Based on these foundations, calculus techniques, such as derivatives and power series, are applied to trigonometric functions \cite{blaszczyk2025}.

\pvn 4. Nowadays, Euclidean proportion, in a more accessible form, is commonly encountered in school mathematics or historical contexts. It is typically expressed using fractions rather than its original  formulation. 
 In fact, the foundational achievement of Descartes' \textit{La Géométrie} (1637) \cite{descartes1637} was the transformation of proportion into the arithmetic of line segments through the implicit rule \cite{blaszczyk2024}:
\[a:b::c:d\Rightarrow a=b\cdot {\frac cd}.\]

In a definition of sorts, Descartes introduced the product and division of line segments on the very first pages of his essay through a diagram: given $AB=1$, the line $EB$ is the product of $DB$ and $CB$; similarly, $CB$ is the result of dividing $EB$ by $DB$; see Fig. \ref{fig0}.  

Obviously, these definitions are not rigorous by modern standards.  Rather, they can be interpreted as an application of Thales' theorem:  instead of the proportion
$DB:1::EB:CB$, Descartes introduces a novel operations, namely $EB=DB\cdot CD$ and 
$BC=\frac{EB}{DB}$.  

Throughout \textit{La Géométrie}, these operations, along with the addition of line segments, satisfy the laws of an ordered field. These rules were applied in mathematics implicitly until the end of the 19th century when Hilbert introduced the axioms of an ordered field \cite{ref_DH99,hilbert1900}.

Thus, alongside trigonometry -- presented either in elementary form or as power series -- Euclidean proportion resonates in modern calculus through the laws of an ordered field.

\begin{figure}[!ht] 
\centering
\includegraphics[scale=0.75]{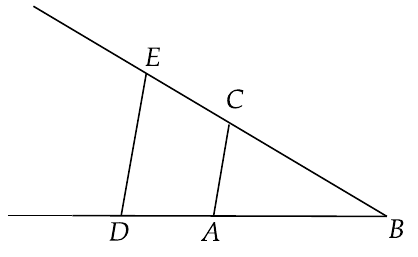}
\includegraphics[scale=0.75]{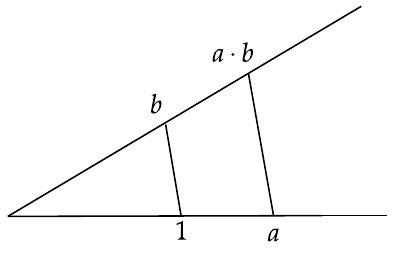}
\caption{\textit{La G{\'e}om}{\'e}trie, p. 298 (left). An interpretation of Descartes' definition (right)} \label{fig0}
\end{figure}

\pvn 4. Moritz Pasch's \textit{Vorlesungen {\"u}ber Neuere Geometrie} (1882) \cite{pasch1882} initiated the process of establishing Euclidean geometry on new foundations. Hilbert's\textit{ Grundlagen der Geometrie} and other 20th-century systems that followed also sought to reconstruct the Euclidean theory of similar figures within these new frameworks.

The ancient concept of proportion was abandoned, giving rise to two general strategies: one based on the arithmetic of line segments and the other on the properties of real numbers, understood as an ordered field with completeness. 

The arithmetic of line segments is inspired by Descartes' arithmetic. Indeed, instead of Descartes' arbitrary angle $\angle DBE$, one can use a right angle while adopting Descartes' definitions; see Fig. \ref{Des1} or \ref{hilbert1}. In this formulation, the rules of an ordered field can be justified within Euclidean proportion; see \S\,3 below. 

However, while Descartes introduced ordered field arithmetic based on the ancient concept of proportion, namely, 
 \[ a\cdot d=c\cdot b \Leftrightarrow_{df}  a:b::c:d,\]
 modern geometry seeks to replicate ancient proportions through the laws of an ordered field. In contemporary frameworks, Euclidean proportion is recovered as the product of line segments:
\[a:b=c:d \Leftrightarrow_{df} a\cdot d=c\cdot b,\]
or within the arithmetic of real numbers:
  \[a:b=c:d \Leftrightarrow_{df} \frac ab=\frac cd.\]

 Below,  we examine variations of the first approach by discussing the systems of Hilbert and Hartshorne in \S\S\,4 and 5, respectively. 
 In this approach, as in Euclid and Descartes, $a, b, c, d$ stand for line segments.

The second strategy relies on real numbers and considers relationships between the lengths of line segments, that is, real numbers assigned to line segments based on axioms or sophisticated arguments. 

Geometry developed in mainstream mathematics interprets Euclidean geometry within the so-called Euclidean spaces $\mathbb R^n$, with $\mathbb R^2$  serving as the model example of the Euclidean plane.    This approach either incorporates a form of completeness for real numbers into the axioms of geometry or assumes a bijection between a geometric line and the real numbers. 
To be clear, this approach seeks to justify, through foundational studies, what really happened when 20th-century mathematics established its foundations on real numbers.

Below, we address real-numbers approaches by discussing the systems of Birkhoff  and Millman–Parker in \S\S\,7 and 8, respectively. In \S\,6, we discuss the system of Borsuk–Szmielew, which forms a bridge between  synthetic, Hilbert-style geometry and geometry based on real numbers. 

Euclid's theory of proportion relies on the Archimedean axiom, which is explicitly included as Definition 4 in Book V. Hilbert's arithmetic of line segments does not reference the Archimedean axiom. On the other hand, approaches based on real numbers do rely on this axiom, as the real numbers form the largest Archimedean field. In \S\,9, we present a proof of Thales' theorem based on the arithmetic of hyperreal numbers, which constitute a non-Archimedean ordered field.
\section {Book VI of the  \textit{Elements}}
\subsection{Definition of proportion}
\pvn 1. Thales' theorem is connected to Euclidean proportion through Proposition VI.1 -- the only proposition in Book VI, except for the last one (VI.33), that explicitly references the definition of proportion,  Defintion 5, Book V. We interpret this definition using the following formula:
  \begin{enumerate}\itemsep 0mm
\item[] $a:b::c:d\Leftrightarrow_{df} (\forall{m,n\in\mathbb{N}})[(na>_1mb\Rightarrow nc>_2md)\wedge\\
 \wedge (na=mb\rightarrow nc=md) \wedge (na<_1mb\Rightarrow nc<_2md)]$.
\end{enumerate}

 Pairs $a, b$ and $c,d$ are to be of the same  \textit{kind}. This assumption  is formalized by $a,b\in{\mathfrak{M}_1}=(M_1,+,<_1)$, and $c,d\in{\mathfrak{M}_2}=(M_2,+,<_2)$, which means that magnitudes of the same kind can be added and compared in terms of greater-than relationship. In the context of the Proposition VI.I, $a, b$ are line segments, and $c, d$ are triangles, or parallelograms. Specifically, based on this foundational assumption, triangles, somehow, can be added and compared as \textit{greater} and \textit{lesser} \cite{blaszczyk2024}. 
 
 In Proposition VI.33, Euclid states the proportion between angles in a circle, on the one hand, and respective arcs, on the other.

\begin{theorem}
[{\textit{Elements}, VI.1} \cite{fitz2008}]
\textit{Let ABC and ACD be triangles, and  EC and CF parallelograms,  of the same height AC. I say that as base BC is to base CD, so triangle ABC (is) to triangle ACD, and parallelogram EC to parallelogram CF.}
\end{theorem}
\begin{figure}[!ht] 
\centering
\includegraphics[scale=0.8]{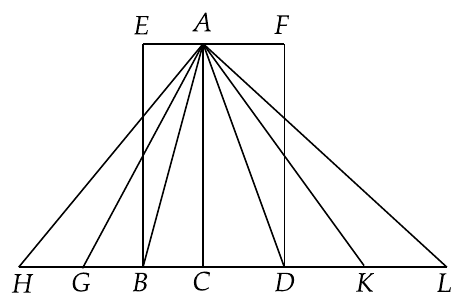}
\caption{\textit{Elements}, VI.1} \label{fig71}
\end{figure}
\begin{proof} 
By  construction:
$BC=GB=HG$ and $CD=DK=KL.$ Then, by I.37, the equality of triangles  holds: 
$\triangle AHC=3\triangle ABC$, and $\triangle ALC=3\triangle ADC$; see Fig. \ref{fig71}

Finally, Euclid applies these data to the definition of proportion:
\begin{equation}\label{VI1}3\triangle ABC \gtreqqless 3 \triangle ADC\Rightarrow  3BC \gtreqqless 3 DC.\end{equation}

 This formula interprets Euclid's words: ``And if base HC is equal to base CL, then triangle AHC is also equal to triangle ACL. And if base HC exceeds base CL, then triangle AHC also exceeds triangle ACL. And if less (than) less.`` 

We express this as:
\[\triangle ABC: \triangle ADC::BC:DC.\]

Note that, instead of \textit{equal multiples} ($na, nc, mb, md$) referred to in Definition V.5, Euclid sets $n=m=3$.
\end{proof}
\pvn 2. The \textit{Elements} do not explicitly justify the reasoning behind why the triangles are equal, greater, or lesser, i.e., (\ref{VI1}). Euclid assumes that the equality or inequality of line segments directly translates to the equality or inequality of the triangles formed on these segments.  Indeed, arguments relying on areas play a role in the theory of similar figures. Modern systems, however, seek to eliminate them.

On the other hand, some crucial results in Book VI concern the relationship between the areas of similar figures, yet modern systems do not recover them.

In Proposition VI.33, Euclid states the proportion between angles in a circle 
(inscribed or central) and their corresponding arcs. The proof similarly applies to triples of angles and triples of arcs; see Fig. \ref{figVI33}. 

Modern geometers, except for Birkhoff, do not address this issue at all, but it still echoes in calculus.

\subsection{The role of VI.2 in Euclid's system and modern mathematics}
\pvn 1. Thales' theorem forms the foundation of the theory of similar figures, which is developed in the subsequent propositions of Book VI of the \textit{Elements}.

Proposition VI.3 states that in a triangle, the bisector of an angle divides the opposite side into segments proportional to the other two sides.

Proposition VI.4 establishes that equiangular triangles are similar, meaning the sides about equal angles are proportional.

The following propositions introduce criteria for similar triangles analogous to the congruence criteria: Side-Side -Side, Side-Angle-Side and Side-Angle-Angle:

VI.5: Triangles with proportional sides are equiangular.

VI.6: If two triangles have equal angles and the sides about these angles are proportional, then the triangles are equiangular, and their corresponding sides are proportional.

VI.7: If two triangles have equal angles and the sides about another pair of equal angles are proportional, then the triangles are similar, provided these angles are either both less than $\pi/2$
 or both greater than or equal to $\pi/2$.

The next two theorems are the most well-known consequences of Thales' Theorem that are not directly related to trigonometry.

VI.8: In a right-angled triangle, the altitude dropped from the right angle divides the triangle into two smaller triangles, each similar to the original triangle and to each other.

VI.9 shows how to divide a line segment 
$AB$ into $n$ equal parts. The procedure is as follows: 
 Let lines 
$AB$ and $AC$ form an angle. Along the arm $AC$, place the same arbitrary segment $n$ times. Connect the endpoint of the last segment to 
$B$. Then, draw $n$  lines parallel to this segment. 
These parallel lines will intersect
$AB$,  dividing it into $n$  equal parts.

In Proposition VI.12, Euclid demonstrates how to find the so-called fourth proportional line segment.

\pv Moreover, Proposition VI.15 forms the basis for the modern formula for the area of a triangle $\frac 12 ab\sin \alpha$, where $\alpha$ is an angle between sides $a$  and $b$. 

Proposition VI.19 introduces the formula -- in modern terms -- stating that the areas of similar triangles are proportional to the square of the similarity scale, while VI.20 extends this result to similar polygons. 

Proposition VI.31 establishes the addition of similar figures. Since the Pythagorean Theorem (\textit{Elements}, I.47) enables the addition of squares, and all squares are similar, we view VI.31 as a generalization of I.47.
\pvn 2.  In discussing modern attempts to prove Thales' theorem, we will show that some of the above theorems are assumed in advance, either as axioms or propositions based on geometric or analytic grounds. 

Specifically, Hilbert proves VI.4, and then, with the use of the arithmetic of line segments, proves VI.2.

Birkhoff adopts VI.6 as an axiom, and then, with the use of the arithmetic of reals numbers, VI.2 follows easily.

Borsuk and Szmielew, as well as Millman and Parker, prove VI.9, and then, using the results concerning measures  (Borsuk and Szmielew) or arithmetic of real numbers (Millman and Parker), they can prove VI.2.

In \S\,10, we present a system in which VI.1 is an axiom and show that this system allows us to prove VI.2 -- indeed, all propositions of Book VI (except VI.33) -- in an Euclidean fashion. This system is also related to a method for automated proofs of propositions from Book VI.

\subsection{Proposition VI.33}
\pvn 1. In Proposition VI.33, Euclid shows that: \textit{In equal circles, angles have the same ratio as the circumferences on which they stand}. 

\begin{figure}
\centering
\includegraphics[scale=0.8]{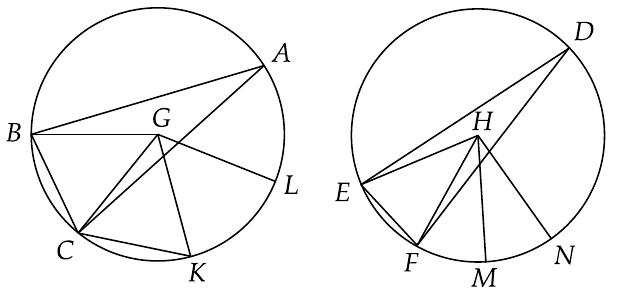}
\caption{\textit{Elements}, VI.33} \label{figVI33}
\end{figure}

The accompanying diagram (see Fig. \ref{figVI33}) represents various magnitudes referred to in the proposition: the angle $\angle BGC$, the sector of the circle $sec\,BGC$, the $arc\,BC$, as well as
the triangle $\triangle BGC$, and the segment $BC$ --  Ptolemy called it 
the chord, and we call it the sine of the angle $\angle BGC$.

With this notation, Euclid's proposition can be phrased as follows:
\[\angle BGC :: \angle EHF ::  arc\,BC: arc\,EF.\]

Since circles are equal, angles and arcs can be taken in the same circle.
Indeed, in modern mathematics, this is typically the unit circle.

\pvn 2. Establishing the relationships between $arc\,BC$ and  
the sine $BC$ was one of the most difficult problems in the history of mathematics.


Ptolemy managed to determine the relationship between the ratio of two arcs and the ratio of two sines. In accordance with Euclidean theory, he compared the ratio of two arcs, which are magnitudes of one kind, with the ratio of two line segments (sines), which are magnitudes of a different kind. This result laid the foundation for tables of chords, which had been in use from antiquity until modern times.

Given that sine $=x$ and arc $=z$  correspond to the angle $\alpha$, Newton determined the arc in terms of sine, that is, the series for $\arcsin x$, as well as the sine in terms of arc, that is, the series of $\sin z$.

Euler managed to combine the series for $\sin x$ and $\cos x$ with the exponential function $e^{ix}$, where $x$ stands for an arc of the unit circle \cite{blaszczyk2023}. 

\pvn 3. In modern mathematics, the identification of the arc $arc\,BC$  and the angle $\angle BGC$  is established through the concept of radian measure. This identification is achieved as follows:
converting degrees $\alpha$ to radians $x$ is based on the formula:  
\[\frac{x}{2\pi} = \frac{\alpha}{360}.\]  

Assuming the unit circle, this formula corresponds to the relationship:  the length of the arc is to the circumference of the circle as the measure of the angle in degrees is to \( 360^\circ \).  That is the straightforward application of Proposition VI.33.

The length of the arc \( l \) corresponding to the angle \( \alpha \) is determined in a similar way:  
\[\frac{l}{2\pi} = \frac{\alpha}{360}.\]  

The area \( P \) of the sector of the circle corresponding to the angle \( \alpha \) is given by the formula:  
\[
\frac{P}{\pi} = \frac{\alpha}{360}.
\]

Hence, we obtain that the length of the arc is equal to the measure of the angle in radians, and the area of the circular sector corresponding to angle \( x \) is \( \frac{x}{2} \):  
\[
x = 2\pi \frac{\alpha}{360}, \quad l = 2\pi \frac{\alpha}{360}, \quad P = \pi \frac{\alpha}{360}.
\]  

\begin{figure}
\centering
\includegraphics[scale=0.8]{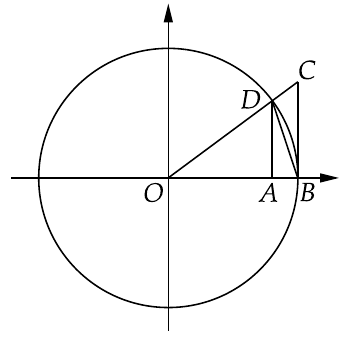}
\caption{Determining inequalities $\sin x < x < \tan x$.}\label{figSinus}
\end{figure}

\pvn 4. Here is how these   relate to determining the limit $\frac{\sin x}{x}$ at $0$.  

In  calculus, the inequalities  
\begin{equation}\label{sin}\sin x < x < \tan x \end{equation}  
are derived by comparing the areas of figures represented in Fig. \ref{figSinus}:  
\[
\text{area of } \triangle ODB < \text{area of the sector of circle } ODB < \text{area of } \triangle OCB.
\]  

 Substituting the formulas for the areas, we obtain  (\ref{sin}).

\subsection{Proving VI.2}

\pvn 1. Below, we reconstruct Euclid's proof of the Thales' theorem.

\begin{theorem}[\textit{Elements}, VI.2 \cite{fitz2008}] If some straight line is drawn parallel to one of the sides of a triangle, then it will cut the sides of the triangle proportionally. And if the sides of a triangle are cut proportionally then the straight line joining the cutting will be parallel to the remaining side of the triangle.
\end{theorem}

\begin{figure}[!ht] 
\centering
\includegraphics[scale=0.8]{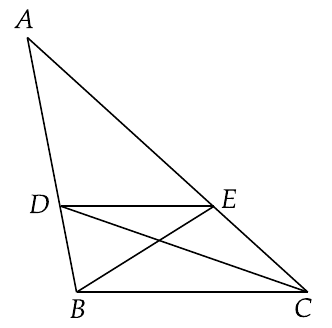}
\caption{\textit{Elements}, VI.2} \label{fig7}
\end{figure}

We present Euclid's proof in a schematized form, using modern symbolic conventions to enhance clarity. Specifically, we employ standard notations such as $\parallel$ to denote parallel lines. Additionally, we introduce specialized symbols, such as $\xrightarrow[I.38]{}$, where the arrow indicates a connective “for” rather than a formal logical implication, and the subscript “I.38” references Euclid’s Proposition I.38. 

The proof is as follows:
 \begin{eqnarray*}
 DE\| BC &\xrightarrow[I.37]{}&\triangle BDE=\triangle CDE\\
&\xrightarrow[V.7]{}& \triangle BDE:\triangle ADE:: \triangle CDE:\triangle ADE\\
 &\xrightarrow[VI.1]{}& \triangle BDE:\triangle ADE::BD:DA\\
&\xrightarrow[VI.1]{}&\triangle CDE:\triangle ADE::CE:EA\\
&\xrightarrow[V.11]{}&BD:DA::CE:EA.
\end{eqnarray*}

The second part goes like that.
  \begin{eqnarray*}
BD:DA::CE:EA, &&\\
BD:DA::  \triangle BDE:\triangle ADE,  &&\\
CE:EA::\triangle CDE:\triangle ADE&\xrightarrow[V.11]{}&\triangle BDE:\triangle ADE::\\
                                &&::\triangle CDE:\triangle ADE\\
&\xrightarrow[V.9]{}&\triangle BDE=\triangle CDE\\
&\xrightarrow[I.39]{}&DE\|BC.
 \end{eqnarray*}

 \begin{figure}[!ht] 
\centering
\includegraphics[scale=0.7]{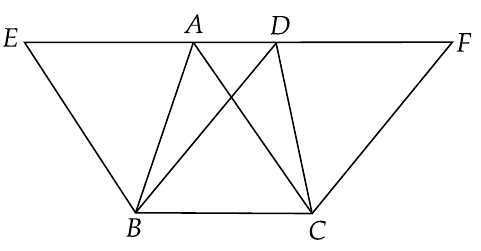}\includegraphics[scale=0.7]{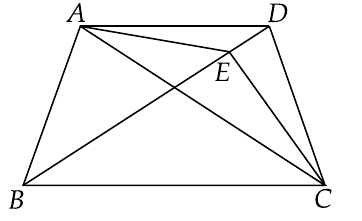}
 \caption{\textit{Elements} I.37 and I.39}\label{figE37}
\end{figure}
\hfill{$\Box$}

Propositions on equal figures referenced in this proof include the following (see Fig. \ref{figE37}):
\begin{theorem}[\textit{Elements}, I.37 \cite{fitz2008}]
    Triangles which are on the same base and between the same parallels are equal to one another.
\end{theorem}
That is, 
\[AD\parallel BC\Rightarrow \triangle ABC=\triangle DBC.\]
\begin{theorem} [\textit{Elements}, I.39 \cite{fitz2008}]
    Equal triangles which are on the same base, and on the same side, are also between the same parallels.
\end{theorem}
 That is, 
\[\triangle ABC=\triangle DBC \Rightarrow AD\parallel BC.\]

\pvn 2. In fact,  VI.2 consists of two propositions: VI.2a, which moves from parallelism to proportion, and VI.2b, which moves from proportion to parallelism.

In a more synthetic manner, supported by Fig. \ref{figVI2a}, the first part of Euclid's proof is as follows:
\[l\parallel p \xrightarrow[I.37]{} T_1=T_2 \xrightarrow[V.7]{} \frac{T_1}{T}=\frac{T_2}{T} \xrightarrow[VI.1]{}\frac{b}{a}=\frac{d}{c}.\]

And the second:
\[\frac{b}{a}=\frac{d}{c} \xrightarrow[VI.1]{} \frac{T_1}{T}=\frac{T_2}{T}\xrightarrow[V.9]{} T_1=T_2\xrightarrow[I.39]{} l\parallel p. \]

In these formulas, $T, T_1$, and $T_2$ represent triangles, and instead of the proportion $a:b::c:d$, we use the equality of fractions.

 \begin{figure}[!ht] 
\centering
\includegraphics[scale=0.8]{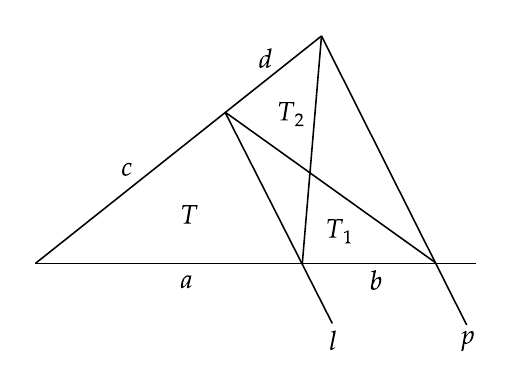}
 \caption{Proof of \textit{Elements}, VI.2 schematized}\label{figVI2a}
\end{figure}

These schemes emphasize the role of three sub-theories in Euclid's proof: the   theory of proportion, Proposition VI.1 (along with the definition of proportion), and the concept of parallel lines related to equal figures through Propositions I.37 and I.39.

\pvn 3. Comparing Euclid's approach with that of modern mathematics, the key issue lies in the ancient concept of proportion. In Euclid's system, it applies to triangles and line segments, while the crucial move relies on the relationship:  
\[\frac{T_1}{T}=\frac{T_2}{T} \Leftrightarrow\frac{b}{a}=\frac{d}{c}.\]

In contrast, Hilbert and Hartshorne reconstruct the proportion of line segments based on the arithmetic of line segments. Using geometric principles, they prove Euclid's Proposition VI.4.   Then, in proving VI.2a,  the assumption of parallel lines $l\parallel p$ implies that the respective triangles are
 equiangular, and, due to VI.4, the proportion $a:c=b:d$ holds; see Fig. \ref{figVI2b} (left).

 Borsuk and Szmielew, as well as Millman and Parker, derive VI.2a from VI.9 -- indeed, it can be proved without referring to Thales' theorem.

\pvn 4. Although the derivation of VI.2 in these systems differs, they enable us to prove VI.2b in the same manner:  Supposing that $\frac ab=\frac cd$ and that $l$ is not parallel to $p$, a line $q$ parallel to $l$ is introduced, which intersects a line segment $d'$. Due to VI.2a,  the proportion  $\frac ab=\frac c{d'}$ holds. By the arithmetic of line segments, it follows that $d=d'$,  leading to a contradiction; see Fig. \ref{figVI2b} (right).

\begin{figure}[!ht] 
\centering
\includegraphics[scale=0.8]{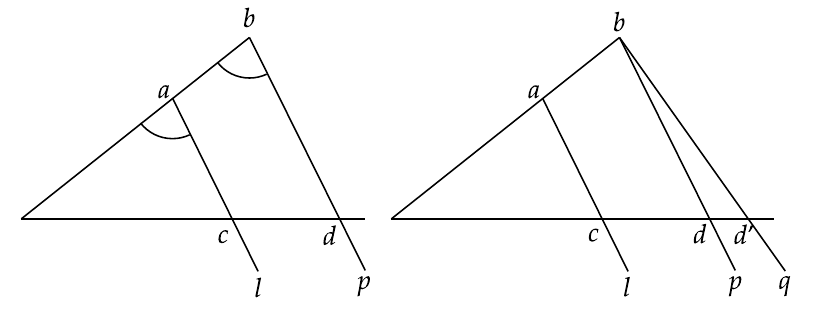}
 \caption{Modern proofs of VI.2a (left) and VI.2b (right) schematized.}\label{figVI2b}
\end{figure}

\pvn 5. In Birkhoff's system, the deductive structure of VI.2 is quite different. Birkhoff adopts VI.6 as an axiom, and from proportionality of sides $\frac ba= \frac dc$,  it follows that $b=ka$ and $d=kc$, where $k=\frac ba$. By VI.6, the respective triangles are equiangular, leading to the conclusion  $l\parallel p$, that is VI.2b; see Fig. \ref{figVI2b} (left).  

For VI.2a, suppose $l\parallel p$ and 
$\frac ba\neq \frac dc$. For some $d'$, the equality holds $\frac ba= \frac {d'}c$.\footnote{In Greek mathematics, this property is called the fourth proportional. It is employed implicitly in Euclid's Book V.} 
Then, $b=ka$ and $d'=kc$, where $k=\frac ba$, and the respective triangles are equiangular, which means $l\parallel q$,  contradicting Playfair Axiom; see Fig. \ref{figVI2b} (right).

After this overview of the techniques applied in proving Thales' theorem, we proceed to a more detailed presentation of specific approaches.

\begin{figure}[!ht] 
\centering
\includegraphics[scale=0.8]{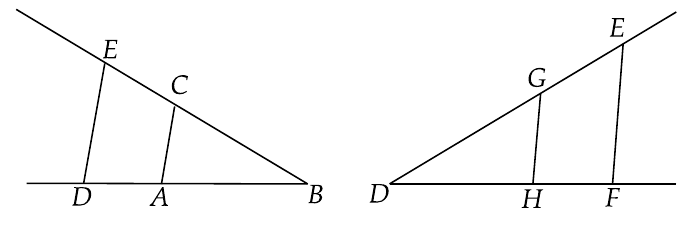}
\caption{Product of line segments: Descartes' \textit{La G{\'e}om}{\'e}trie, p. 298 (left), Euclid's \textit{Elements}, VI.12 (right).} \label{Des0}
\end{figure}

\section{Descartes' arithmetic}

\subsection{\textit{Elements}, Book V} 

Descartes' arithmetic of line segments is based on Proposition VI.12 of Euclid's \textit{Elements}. Descartes adopts a unit line segment $AB=1$ as well as an arbitrary angle $\angle DBE$; see Fig. \ref{Des0}. To simplify the setup, we assume $\angle DBE=\frac\pi 2$ and demonstrate that the rules of arithmetic can be justified by the laws of proportion developed in Book V of the \textit{Elements}.

Below, we include Propositions 7 to 25 of Book V. Although they are stylized in algebraic form, this modern formulation serves only to highlight the similarities between proportions and the arithmetic of fractions. Here, equality denotes equal figures \cite{ref_BP}.

\pv V.7 \tab $a=b\rightarrow a:c::b:c, \;\; a=b\Rightarrow c:a::c:b.$
\pv V.8 \tab$a>c\Rightarrow{a:d\succ c:d},\;\; a>c\Rightarrow d:c\succ d:a.$
\pv V.9 \tab$a:c::b:c\Rightarrow{a=b}.$
\pv V.10  \tab$a:c\succ b:c\Rightarrow{a>b},\;\;  c:b\succ c:a\Rightarrow{b<a}.$
\pv V.11 \tab$a:b::c:d,\; c:d::e:f\Rightarrow{a:b::e:f}$
\pv V.12 \tab$a:b::c:d,\; a:b::e:f\Rightarrow{a:b::(a+c+f):(b+d+f)}.$
\pv V.13 \tab$a:b::c:d,\ c:d\succ e:f\Rightarrow a:b\succ e:f.$
\pv V.14 \tab$a:b::c:d,\; a> c \Rightarrow b>d.$
\pv V.15 \tab$ a:b::na:nb.$
\pv V.16 \tab$a:b::c:d\Rightarrow{a:c::b:d}.$
\pv V.17\tab$(a+b):b::(c+d):d\Rightarrow{a:b::c:d}.$
\pv V.18  \tab$a:b::c:d\Rightarrow{(a+b):b::(c+d):d}.$
\pv V.19 \tab$(a+b):(c+d)::a:c \Rightarrow b:d::(a+b):(c+d).$
\pv V.22  \tab$a:b::d:e,\; b:c::e:f \Rightarrow a:c::d:f.$
\pv V.23 \tab$ (a:b::e:f, \; b:c::d:e) \Rightarrow a:c::d:f.$
\pv V.24 \tab$a:c::d:f,\; b:c::e:f \Rightarrow (a+b):c::(d+e):f.$
\pv V.25 \tab$(a:c::e:f,\; a>c>f,\; a>e>f)\Rightarrow a+f>c+e.$

\begin{figure}[!ht] 
\centering
\includegraphics[scale=0.8]{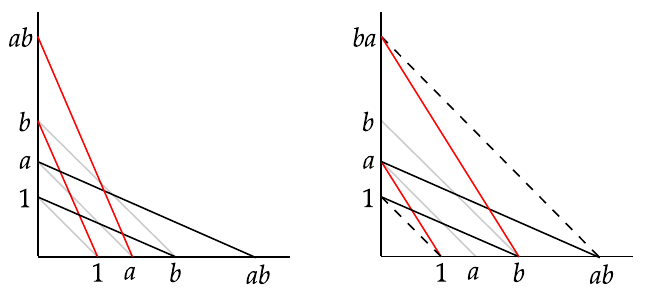}
\caption{Product of line segments (left) Commutativity of the product (right). } \label{Des1}
\end{figure}
\subsection{Arithmetic of line segments}

First, we demonstrate the commutativity of the product. 
By the definition of the product, we obtain 
(see red lines in Fig. \ref{Des1}):  
$$\frac a1=\frac {ab}b,\ \ \ \frac{ba}a=\frac b1.$$

Then, by  Proposition V.23, the following proportion holds:
\[\frac a1=\frac {ab}b,\ \frac{ba}a= \frac b1 \xrightarrow[V.23]{} \frac {ba}1=\frac {ab}1.\]

By Proposition V.9, it follows that
\[ba=ab.\]
\hfill{$\Box$}

In Fig. \ref{Des1} (left), we represent product $ab$ 
on two axes, given the congruence of the respective right-angled triangles.
In Fig. \ref{Des1} (right), we represent the product $ab$ as a result of this modification. 
 Note that the dashed lines are parallel due to a result derived from the theorems in Book V: since $ba=ab$, the respective triangle is isosceles.

 Conversely, in Hilbert's arithmetic of line segments and other modern approaches to the arithmetic of line segments, geometric arguments establish that certain lines are parallel, thereby justifying that $ab=ba$.

Second, the distributive law requires more careful attention.
The left diagram in Fig. \ref{Des2} illustrates the products \( c \cdot a \) and \( c \cdot b \); this representation assumes that the product is commutative. By drawing the parallel to the line $c1$ through $a+b$, we apply Thales' theorem to obtain  
\[\frac xa =\frac{ca} a,\]
where the segment \( a \) in the ratio \( x : a \) corresponds to the difference \( (a+b) - b \). 

Applying Euclid's Proposition V.9, we conclude that \( x = ca \). 

Ultimately, this leads to  
\[
x + cb = ca + cb.
\]

The continuous lines in the right diagram in Fig. \ref{Des2} represent the following proportions, expressed as fractions (the first follows from the definition of the product, and the second from the above argument):
\[ \frac{c}{1} = \frac{c(a+b)}{a+b}, \quad \frac{ca+cb}{c} = \frac{a+b}{1}.\]

Applying Proposition V.23, we obtain:  

\[
\frac{ca+cb}{1} = \frac{c(a+b)}{1}.
\]

Finally, applying Proposition V.9 once again, we arrive at:
\[
ca + cb = c(a+b),
\]
where \( ca \), \( cb \), and \( c(a+b) \) are constructed based on Descartes' definition.

\hfill{$\Box$}

Fig. \ref{Des1} and \ref{Des2}  depict the respective relations between parallel lines. However -- let us emphasize -- the arguments presented above do not rely on geometry but rather on propositions from Book V. Although modern geometers adopt Descartes' definition of the product, they seek to justify the laws of arithmetic through geometric principles. 

Dashed lines in diagrams Fig. \ref{Des1} and \ref{Des2} are parallel due to the theory of proportions. Based on the above arguments, Euclid's Proposition  V.23 
encodes the commutativity and distributivity of Descartes' product. 

\begin{figure}
\centerline{\includegraphics[scale=0.8]{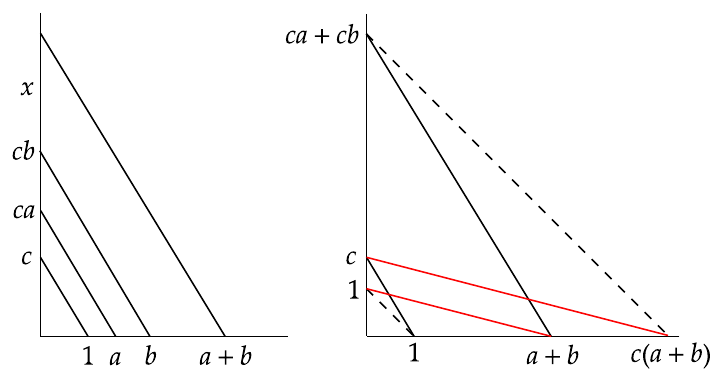}}
\caption{Distribuitive law.}
\label{Des2}
\end{figure}
\section{Hilbert}

Hilbert's axioms for synthetic geometry were first presented in \textit{Grundlagen der Geometrie} (1899) \cite{ref_DH99}. Although referring to Euclid's \textit{Elements}, they are based on a different methodology. Hilbert does not include a straightedge and compass as constructive tools. From the set of constructions developed by Euclid, he selects two specific ones: the transportation of line segments (\textit{Elements}, I.2, \textit{Grundlagen}, C1) and the transportation of angles (\textit{Elements}, I.23, \textit{Grundlagen}, C4). Accordingly, the deductive structures of the \textit{Elements} and \textit{Grundlagen} differ \cite{blaszczyk2021}.

Moreover, a new approach, unknown in Greek mathematics, is based on the uniqueness referenced in Axioms I.1, C.1, and C.4; the Parallel Axiom also implies uniqueness.

In what follows, we will point out only the differences related to Thales' theorem.

Below, we present Hilbert's axioms following Marvin Greenberg's concise version \cite{greenberg2008}.

\subsection{Hilbert system of axioms}

Hilbert adopts primitive concepts: point, line, and plane, as well as   primitive relationships:
point $B$ lies between $A$ and $C$, denoted here as $A-B-C$, congruence of line segments and angles. He also defines a half-line (ray), an angle, and a triangle.

{Axioms of Incidence}

I1. For any two distinct points A, B, there exists a~unique line $l$ containing $A$,~$B$. 

I2. Every line contains at least two points. 

I3. There exist three noncollinear points (that is, three points not all contained 
in a~single line). 

{Axioms of Betweenness }
 
B1. If $B$ is between $A$ and $C$, (written $A-B-C$), then $A$, $B$, $C$ are three distinct points on a~line, and also $C-B-A$. 

B2. For any two distinct points $A$, $B$, there exist  points $C$, $D$, $E$ such that $A-B-C$, $A-D-B$, and
$E - A - B$.

B3. Given three distinct points on a~line, one and only one of them is between 
the other two. 

B4. (Pasch). Let $A$, $B$, $C$ be three non collinear points, and let $l$ be a~line 
not containing any of $A$, $B$, $C$. If $l$ contains a~point $D$ lying between $A$ and $B$, then it must also contain either a~point lying between $A$ and $C$ or a~point lying between $B$ and $C$. 
 
{Axioms of Congruence for Line Segments}

C1. Given a~line segment $AB$, and given a~ray (half-line) $r$ originating at a~point $C$, there 
exists a~unique point $D$ on the ray $r$ such that $AB \equiv CD$.

C2. If $AB \equiv CD$ and $AB \equiv EF$, then $CD \equiv EF$. Every line segment is congruent to itself.

C3. (Addition). Given three points $A$, $B$, $C$ on a~line satisfying $A - B - C$, and 
three further points $D$, $E$, $F$ on a~line satisfying $D - E - F$, if $AB \equiv DE$ and 
$BC \equiv EF$, then $AC \equiv DF$.

{Axioms of congruence for Angles}

C4. Given an angle $\angle BAC$ and given a~ray $\overrightarrow{DF}$, there exists a~unique ray $\overrightarrow{DE}$, on a~given side of the line $DF$, such that $\angle BAC \equiv \angle EDF$. 

CS. For any three angles $\alpha, \beta, \gamma$, if $\alpha \equiv \beta$ and $\alpha \equiv \gamma$, then $\beta \equiv \gamma$. Every angle is congruent to itself. 

C6. (SAS) Given triangles $ABC$ and $DEF$, suppose that $AB \equiv DE$ and $AC \equiv DF$, and $\angle BAC \equiv \angle EDF$. Then the two triangles are congruent, namely, $BC \equiv EF$, $\angle ABC \equiv \angle DEF$ and $\angle ACB \equiv \angle DFE$. 

{Archimedes' axiom}
 
 Given line segments $AB$ and $CD$, there is a~natural number $n$ such that $n$ 
copies of $AB$ added together will be greater than $CD$. 

{Parallel axiom}

 For each point $A$ and each line $l$, there is at most one line containing $A$ that is parallel to $l$.

\subsection{Hilbert's  arithmetic of line segments}

\pvn 1. Addition and the greater-than relation between line segments and angles are defined in Hilbert’s system, whereas in the \textit{Elements}, these are a primitive operation and a primitive relationship, respectively. Today, Hilbert-style definitions are standard. Here is a reminder:
\begin{definition} \label{ad_hi} \cite[p.30]{dh1950} We
say that $c = AC$ is the sum of the two segments $a = AB$ and $b = BC$ if $B$ lies between $A$ and $C$. In other words
\[c=a+b\Leftrightarrow_{df} A-B-C.\]
\end{definition}

 The segments $a$ and $b$ are said to be smaller than $c$, which we indicate by writing
\[ a<c,\; b<c\Leftrightarrow_{df} A-B-C.\]

Addition is both associative and commutative.

\pvn 2. Thales' theorem is presented in \textit{Grundlagen}  as the final result in the chapter entitled \textit{Theory of Proportion}.
Hilbert begins this chapter by saying: ''At the beginning of this chapter, we shall present briefly certain preliminary ideas concerning complex number systems, which will later be of service to us in our discussion.`` He then enumerates 17 properties related to addition, multiplication, and the greater-than relationship. In fact, he provides the first-ever axioms for an ordered field.

In the following, Hilbert clarifies: ''In the present chapter, we propose, by the aid of these axioms, to establish Euclid’s theory of proportion; that is, we shall establish it for the plane and that independently of the axiom of Archimedes``.

Indeed, the definition of the product of line segments, as well as its properties such as commutativity and distributivity, is based on the properties of congruent triangles and angles inscribed in a circle, as presented in Books I and III of \textit{The Elements}. The Archimedean axiom is introduced as Definition 4 of Book V, meaning that Euclid's theory of proportion and similar triangles depends on it.

\begin{theorem}[{Pascal’s theorem}]\cite[p.25]{dh1950}

 Let  $A$, $B$, $C$ and $A'$, $B'$, $C'$ be two sets of points on the arms of an angle. If $CB'$ is parallel to $BC'$ and $CA'$ is  parallel to $AC'$, then $BA'$ is parallel to $AB'$.
\end{theorem}

\begin{figure}
\centering
\includegraphics[scale=0.9]{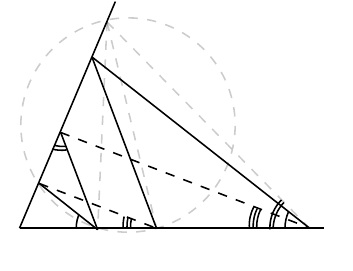} \, \, \, \includegraphics[scale=0.9]{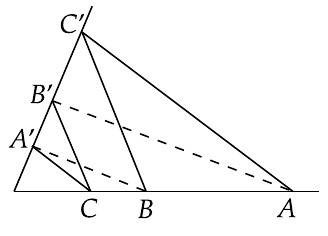}
\caption{Proof of Pascal's theorem}\label{Pascal}
\end{figure}

 The proof of this theorem is quite intricate, prompting Hilbert to dedicate an entire subsection to its discussion. Figure \ref{Pascal} illustrates the proof. Essentially, Hilbert explores the reverse of Euclid's III.22 on a quadrilateral inscribed in a circle.
 
 The definition of the product of line segments aligns with Descartes' approach, using the right angle instead of any angle. Moreover, Hilbert
 modifies Descartes' approach setting $a$ and $ab$ on the same arm of 
 the angle; see Fig.\ref{hilbert1} (left) and Fig. \ref{Des1}.

\begin{figure}[!ht] 
\centering
\includegraphics[scale=0.8]{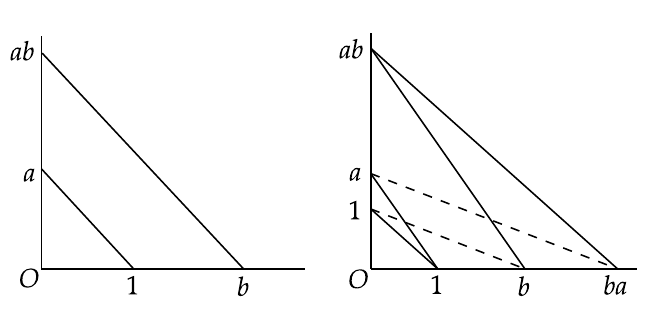}
\caption {Hilbert's definition of the product of (left). Commutativity of the product (right).} \label{hilbert1}
\end{figure}

Hilbert demonstrates that this product is commutative, associative, and distributive. The proofs are based on Pascal's theorem. Figure \ref{hilbert1} (right) illustrates the proof of the commutativity law: Since triangles $\triangle 1O1$ and $\triangle (ab)O(ba)$ are 
isosceles, the equality $ab=ba$ follows.

\pvn 3.  Hilbert axioms for an ordered field include the following: For $a\neq 0$, and $b$, there exists the unique element $x$ such that
\[ax=b.\]

 Figure \ref{hilbert7} (left) illustrates the construction of the line $\frac ba$.
The same figure shows that $\frac ba a=b$. Since multiplication is commutative, it follows that $\frac ba$ solves the equation $ax=b$.

Setting $b=1$ in this axiom implies the existence of the inverse $a^{-1}$; see Fig. \ref{hilbert7} (right).

\begin{figure}[!ht] 
\centering
\includegraphics[scale=0.8]{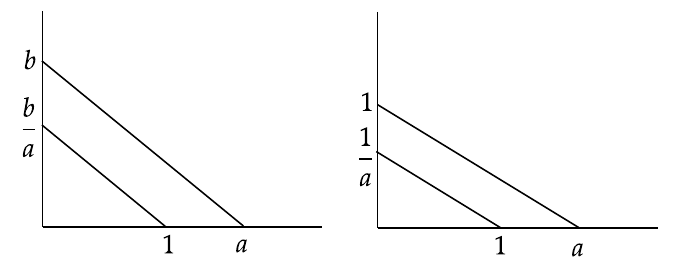}
\caption {Division of line segments (left). The inverse of a line segment (right).} \label{hilbert7}
\end{figure}

\subsection{Thales' theorem and similar triangles}

\pvn 1. In Descartes' \textit{La Geometrie}, we identify the implicit rule that transforms proportion into an equality of quotients \cite{blaszczyk2024}:
\[a:b::c:d\Leftrightarrow \frac ab=\frac cd.\]

However, just as Descartes' arithmetic is based on Thales' theorem, so is this rule. Hilbert arrived at a similar result without referring to Thales. 

\begin{definition} \cite[p. 34]{dh1950}
Segments $a$, $b$, $a'$, $b'$  are in proportion if and only if the equality of products $ab'$ and $ba'$ holds:
\[a:b=a':b'\Leftrightarrow_{df} ab'=a'b.\]
\end{definition}

Since Hilbert's arithmetic includes division, we can rephrase this as follows 
\[a:b=a':b'\Leftrightarrow \frac ab=\frac {a'}{b'}.\]
\begin{definition}\cite[p.34]{dh1950}
Equiangular triangles are called similar.
\end{definition}

\pvn 2. Before addressing Thales' theorem, Hilbert proves the counterpart of Euclid's Proposition VI.4.

\begin{theorem} \cite[p. 34]{dh1950} \label{prop_Hil1}
In equiangular triangles sides about equal angles are proportional.
\end{theorem}

\begin{proof} Hilbert first considers right-angled triangles and then, in the general case, examines triangles composed of right-angled triangles; 
see Fig. \ref{hilbert2} and Fig. \ref{hilbert3}.

\begin{figure}[!ht] 
\centering
\includegraphics[scale=0.8]{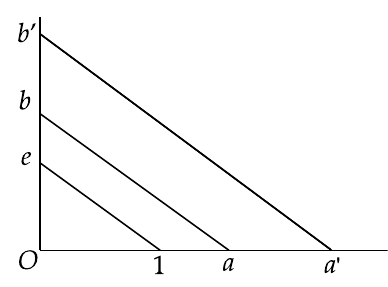} 
 \caption {The first part of proposition \ref{prop_Hil1}} \label{hilbert2}
\end{figure}

The assumption about equiangular triangles $\triangle bOa$ and $\triangle b'Oa'$ implies that lines through $a, b$ and $a', b'$ are parallel. A point $e$ is chosen so that the line passing through $e, 1$ is parallel to these lines. 

By the definition of the product, 
we have $b = ea$ and $b' = ea'$. Then, $a'b = a'ea$, $ ea'a=b'a$. Due to commutativity, 
$a'b = b'a$, that is, $a : b = a' : b'$.

In the general case, Hilbert proceeds as follows: intersecting the bisectors in triangles $T$ and $T'$ determines points 
$S$ and $S'$. Perpendiculars dropped from $S$ and $S'$ 
 onto the sides of the triangles decompose  $T$ and $T'$ 
  into right-angled triangles such that respective pairs are equiangular. 
  This approach allows Hilbert to apply the previous result related to right-angled triangles; see Fig. \ref{hilbert3}.

 
\begin{figure}[!ht] 
\centering
\includegraphics[scale=0.8]{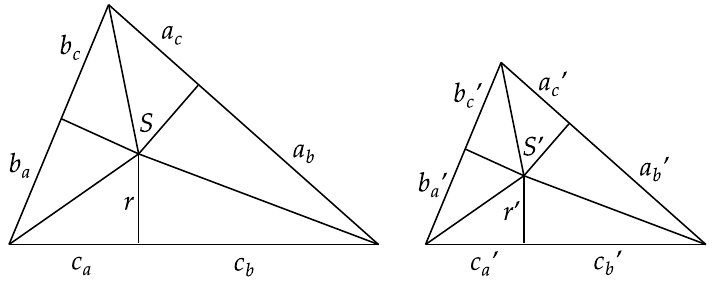}
 \caption {Decomposition of equiangular triangles.} \label{hilbert3}
\end{figure}

Consequently, proportions relating to respective triangles making $T$ and $T'$ follow: 
\[a_b : r = a_b': r',\ \ \ a_c : r = a_c': r', \ \ \
 b_c : r = b_c': r', \ \ \ b_a : r = b_a' : r'.\]


Turning proportions into equalities of products gives:
\[a_b  r' = a_b' r,\ \ \ a_c r' = a_c'r, \ \ \
 b_c r' = b_c'r, \ \ \ b_a  r' = b_a'  r.\]

By the distributive law, one obtains:
 
$$(a_b+a_c)r'= (a_b'+a_c')r',\;\;\;\;\; (b_c+b_a)r'= (b_c'+b_a')r',$$

Then:
\[a  r' = a'  r, \ \ \ \  b  r' = b'  r,\]

and also:
$$b'a  r' = b'a'  r, \;\;\;\;\; a'b  r' = a'b'  r.$$

Finally, dividing the equality $b'a  r' =  a'b  r'$  by $r'$ gives
$$b'a =  a'b, $$

or, in the equivalent form:
$$a : b = a' : b'.$$

Note that, to divide $b'a  r' =  a'b  r'$  by $r'$, Hilbert had to show that multiplication of line segments is an associative operation.
\end{proof}

Furthermore Hilbert writes: ``From the theorem just demonstrated, we can easily deduce the fundamental theorem in the theory of proportion''.  

\begin{theorem} \cite[p. 35]{dh1950}\label{prop_Hil2}
If two parallel lines cut segments
$a$, $b$    from one side of an angle and segments $a'$, $b'$ from the other side, then the proportion $a : b = a' : b'$ holds.
Conversely, if  four segments $a$, $b$, $a'$, $b'$
satisfy this proportion 
and $a$, $a'$ and
$b$, $b'$ 
are laid off along the two sides of an angle respectively, then the straight lines joining the extremities of
 $a$
and $b$ and of $a'$ and $b'$ are parallel.
\end{theorem}

 Hilbert leaves this without proof. Indeed, the proof proceeds straightforwardly, if  instead of proportion 
 $a:b=a':b'$, we consider the equality:
 \[ \frac ab=\frac{a'}{b'}.\]

Then, we can apply the rule of arithmetic, specifically:\footnote{In fact, it is an arithmetic interpretation of Euclid's Proposition V.17. }  
\begin{equation}\label{ratio} \frac a{a'}=\frac {a+b}{a'+b'}\Rightarrow 
   \frac a{a'}=\frac {b}{b'}.\end{equation}
 
  As triangles $\triangle BAC$ and $\triangle BDE$ are equiangular, it follows from the previous theorem and (\ref{ratio}) that
  $$\frac a{a'}=\frac {b}{b'}.$$

\begin{figure}[!ht] 
\centering
\includegraphics[scale=0.7]{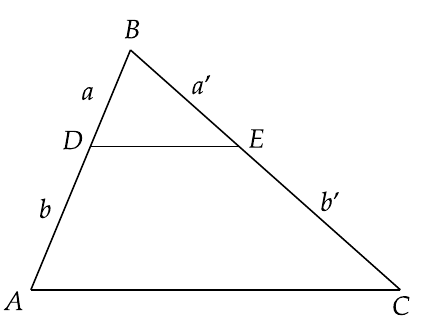} \includegraphics[scale=0.7]{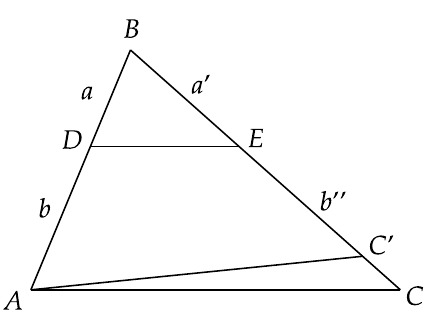}
 \caption {Proposition \ref{prop_Hil2}} \label{hilbert4}
\end{figure}

For the second part, suppose $\frac ab = \frac{a'}{b'}$ and $DE\nparallel AC$.
Let the line $AC'$ parallel to $DE$ cut on the side $BC$ line segment $b''$. 
By the Playfair's axiom, the segment $b"$ is unique.\footnote{For clarity, $b'=BC$ and $b''=BC'$.}  Now, let $b''<b'$; see Fig. \ref{hilbert4} 
  
  By the first part, the equalitites hold:
\[\frac {a'}{b'}=\frac ab=\frac {a'}{b''}.\]

This leads to a contradiction, as it implies $b'=b''$. 

\hfill{$\Box$}

\pv Note that in Euclid's system, the implication
\[a:b::a:b''\Rightarrow b'=b''\]
also holds but is based on Proposition V.9, which requires Archimedes axiom. 

On the other hand, 
\[\frac {a'}{b'}=\frac {a'}{b''}\Rightarrow  b=b''\]
obtains in any ordered field (Archimedean or non-Archimedean) and, therefore, in Hilbert's arithmetic of line segments.

\section{Harsthorne}

 \subsection{Two ways of introducing product}

 \begin{figure}[!ht] 
\centering
\includegraphics[scale=0.7]{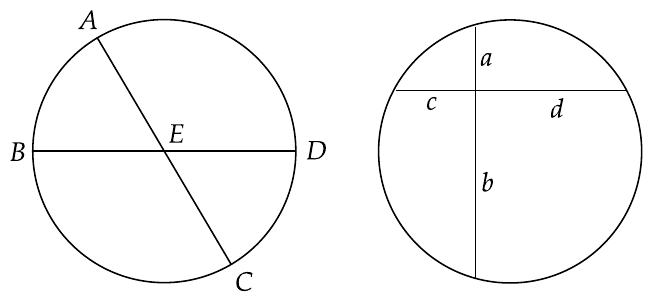}
\caption{\textit{Elements}, III.35 (left) and its simplified version (right).} \label{EuIII35}
\end{figure}
 \pvn 1. There are two propositions in Euclid's \textit{Elements} that may serve as a basis for the product of line segments: VI.12 and III.35. Descartes and later Hilbert followed the first path, while Robin Hartshorne \cite{ref_RH} chose the other.

 Hilbert's axioms do not include circles; however, his arithmetic employs properties of quadrilaterals inscribed in a circle. Hartshorne simplified these arguments by introducing the arithmetic of line segments based on Euclid's Proposition III.35.

Euclid states III.35 in terms of equal areas:  Chords in a circle intersect in such a way that the respective rectangles are equal, 
    \[AE.EC=BE.ED\] 
    
    In a simplified form, we can consider perpendicular chords, 
    giving\footnote{Here, in Euclidean context, $AE.BE$ or $ab$ denotes a rectangle with sides $AE, EC$ or $a, b$.}
    \[ab=cd;\]  see. Fig. \ref{EuIII35}.      When one of the lines $a, b, c, d$ is set to $1$, say $c=1$, we may treat $d=ab$ as a definition of a product.

 \begin{figure}
\centering
\includegraphics[scale=0.8]{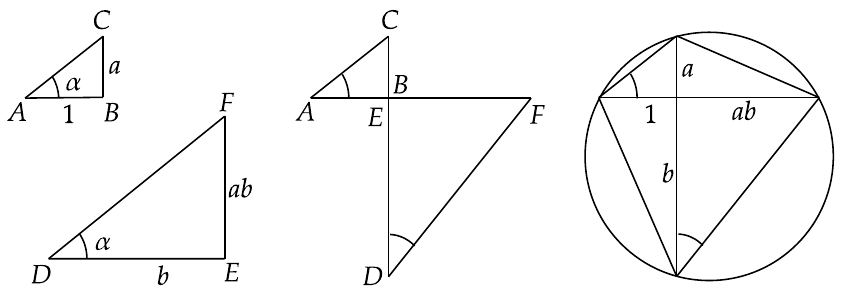} 
\caption{Hartshorne's definition of the product of line segments.}\label{hartsh6}
\end{figure}
 
  Hartshorne's definition of a product is as follows (\cite{ref_RH}, p. 170): Taking two equiangular right-angled triangles — the first with legs $1$ and $a$, the second with leg $b$, the line $EF$ is the product of $CB$ and $ED$; see Fig. \ref{hartsh6} (left). Indeed, given that these triangles are arranged as shown in Fig. \ref{hartsh6} (middle), vertices $A, C, D$, and $E$  form a cyclic quadrilateral, and we can interpret the lines $AF$ and $CD$ as intersecting chords; see Fig. \ref{hartsh6} (right).

 \pvn 2. In Fig. \ref{hartsh5}, we compare two ways of introducing the product based on Euclid's propositions. The assumption $l\parallel p$, translates into the equality of angles, $\alpha=\beta$, and then, by VI.2a, into the proportion $a:c::d:b$. On the other hand, the proportion $a:c::d:b$, by VI.2b, translates into the condition $l\parallel p$. We summarize this in the following formula:
\[\alpha=\beta   \xLeftrightarrow[VI.2]{} \frac ac=\frac db.\]

In the second approach, if $\alpha=\beta$, 
then the vertices of the triangles form a cyclic quadrilateral,
 and by III.35, $ac=bd$. The implication can also be reversed, but the corresponding argument requires a new setting involving equal areas.  We summarize this in the following formula:
\[\alpha=\beta   \xLeftrightarrow[III.35]{} ac=bd. \]

The definition of multiplication of segments in the above manner allows Hartshorne to simplify the proof of the properties of multiplication.

\begin{figure}[!ht] 
\centering
\includegraphics[scale=0.8]{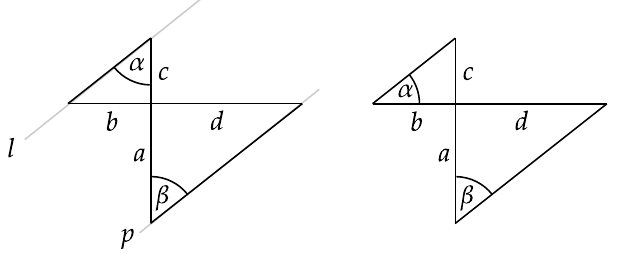}
\caption {Two ways of defining product: proportion (left) and intersecting chords (right).} \label{hartsh5}
\end{figure}
\subsection{Arithmetic and Thales' theorem}
 Hartshorne's approach simplifies proofs of properties such as commutativity and distributivity. For example, to show that $ab=ba$, we consider a triangle with angle $\alpha$ and determine the product $ab$. In the same circle, if we consider a triangle with angle $\beta$,  we obtain the product  $ba$.  It turns out to be the same line segment. Therefore, $ab=ba$; see  Fig. \ref{hartsh7} (left).

The inverse of a line segment is also easily determined within this approach. An isosceles triangle with height  $b$ and base $2$ determines a circle, which in turn determines the segment $\frac 1b$;  Fig. \ref{hartsh7} (middle and right).

A similar technique enables Hartshorne to prove that the multiplication of line segments is an associative operation   \cite[p. 172]{ref_RH}. Therefore, his arithmetic enables one to find the fourth proportional. Indeed, his proof of Thales' theorem follows Hilbert's theorems \ref{prop_Hil1} and \ref{prop_Hil2}, presented above.

\begin{figure}
\centering
\includegraphics[scale=0.8]{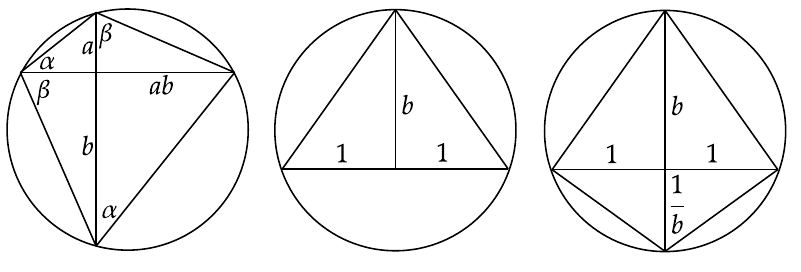} 
\caption{Hartshorne's arithmetic of line segments.}\label{hartsh7}
\end{figure}

\section{Borsuk-Szmielew}

\subsection{Introducing the measure of line segments}
\pvn 1. In \cite{ref_BSz72}, an absolute plane geometry is defined as 
the system $(S, {P}, {L}, \textbf{B}, \textbf{D})$, where $S$ (space) represents the set of points, ${P}$ the set of planes, ${L}$ the set of straight lines, $\textbf{B}$ the ternary relation of betweenness, and $\textbf{D}$ the quaternary relation of equidistance -- both in the space $S$. These primitive concepts and relations are governed by axioms of incidence, betweenness, and equidistance, and the congruence of line segments. 

Presenting this approach to Thales theorem, we consider only line segments, therefore,  for simplicity, we adopt the notation $a \equiv b$ to denote the congruence of line segments.\footnote{$\textbf{D}$ is the relation between four points and defines a congruence of line segments: $AB\equiv CD$ iff $\textbf{D}(A,B,C,D)$. Strictly speaking,  a line segment $a$ is an equivalence class of congruent segments determined by two points.}

Generally, Borsuk and Szmielew adopt Hilbert's axioms. However, unlike Hilbert's system, the Borsuk-Szmielew axioms do not include the concept of an angle. Thus,  instead of Hilbert's axiom C4 (a transportation of angles), they include an axiom on the transportation of triangles, and instead of C6 (the Side-Angle-Side congruence criteria), they include the so-called five-segments axiom. The parallel axiom remains the same as in Hilbert's system, 
namely, the Playfair's axiom.

\pvn 2. The crucial novelty is the continuity axiom expressed in terms of Dedekind cuts --
specifically, in terms of total order, or more precisely, the relation   $\textbf{B}$ (\cite{ref_BSz72}, p. 140).
This axiom enables the demonstration of the existence of a finitely additive measure on the set of line segments.
\begin{theorem} \cite[pp. 156--157]{ref_BSz72} For a given line segment $a$ and a real number $r$, there exists a unique measure $\varphi$ on the set of line segments, satisfying $\varphi (a)=r$ and the following conditions:
\begin{equation}\label{Measure} (1)\ \ a\equiv b\Rightarrow \varphi(a)=\varphi(b), \ \ (2)\ \ \varphi(a+b)=\varphi(a) +\varphi(b).\end{equation}
    \end{theorem}

 The measure $\varphi(a)$, or the length of the line segment $a$, is denoted by $|a|_\varphi$.  Since there are various measures, for simplicity, we distinguish one that assigns the value 
$1$   to a given line segment, say $u$. We denote this measure as $|a|$. Thus
 \[|u|=1.\]

The crucial property of a measure is as follows: given its value on one line segment, for example, $\varphi(u)$, it determines the value for any other line segment. In other words, if the measure 
$\varphi$ is determined for some line segment $a$, then $\varphi$
extends uniquely on the entire set of line segments.

Clearly, in Euclidean geometry, the choice of the unit line segment 
is based on convention. In contrast, in hyperbolic geometry, there exists a unique line segment that is distinguished on a geometric basis (\cite{ref_BSz72}, p. 242). 

Note also that the real number unit is itself a matter of convention. Given that $(\R,+,\cdot\,,0,1,<)$ is the field of real numbers -- an ordered field equipped with the completeness axiom -- the field
$(\R_+, +_1,\circ,1,e,<_1)$ is also the field of real numbers. In this new field, the (standard) product plays the role of addition, and a new product is introduced via the exponential map:  
\[x+_1y=_{df}  x\cdot y,\ \ x\circ y=_{df}e^{\log x\cdot \log y};\]
while the total order $<_1$ is a restriction of $<$ to the set $\R_+$.

Amid these conventions, the total order of real numbers holds a special position, as it is the unique total order compatible with both addition and multiplication of real numbers.

\pvn 3. The existence measure theorem is quite involved and essentially explores a property of real numbers, specifically, that the set of dyadic numbers is dense in 
$(\R,<)$.\footnote{In the context of ordered fields, this is equivalent to the Archimedean axiom.} 

Similarly, due to the density of dyadic numbers in $(\R,<)$, the following theorem is established, which, as we demonstrate below, essentially encodes Thales' theorem.
\begin{theorem}\label{M2} \cite[p. 156]{ref_BSz72} For any two measures $\varphi_0$ and $\varphi_1$, there exists a real number $\lambda$ such that
\begin{equation}\label{T}\varphi_0=\lambda\varphi_1.\end{equation}
\end{theorem}

The proof of Theorem \ref{M2} involves a special operation on segments introduced by
Borsuk and Szmielew.  Specifically, for any dyadic number 
$\mathfrak w$ and any line segment $a$, they define a new operation:
\[\mathfrak wa,\]
where the line segment $\mathfrak wa$  
is constructed through additions and bisections. Moreover, for any measure function $\varphi$, the following holds:
\[\varphi(\mathfrak w a)=\mathfrak w \varphi(a), \ \ \ \varphi(a-b)=\varphi(a)-\varphi(b),\]
as well as the compatibility with the order:\footnote{Note, however, that in modern systems, the \textit{greater-than} relationship between line segments is defined through the relation \textbf{B}, and its uniqueness is not demonstrated.}
\[a<b\Rightarrow \varphi(a)<\varphi(b).\]

Since a value on a line segment $a$
determines the measure of any other line segment,  we can rephrase Theorem \ref{M2}
as follows: For any measure $\varphi$, there exists a real number $\lambda$ such that
\begin{equation}
    \varphi(a)=\lambda |a|.
\end{equation}

Roughly speaking, given two measures  $\varphi_0$,  $\varphi_1$ such that $\varphi_0 (a)=r$, and  $\varphi_1(a)=s$, the relationship between $\varphi_0$ and $\varphi_1$ is given by
\[ \varphi_0= \tfrac rs\varphi_1. \]

Finally \cite[pp. 156--157]{ref_BSz72}, Borsuk and Szmielew  demonstrate that a line $l$ is isometric to the line of real numbers $(\R,<)$.
In another words, starting from scratch -- using only the axioms of synthetic geometry along with continuity -- they show that any geometric line is isomorphic to the real number line, where the metric on $(\R,<)$ is given by the absolute value of real numbers: 
\[\varrho (r,s)=|r-s|,\ \ \ \ r,s\in\R.\]

Since the real numbers form a real-closed field, there exists the unique order compatible with both addition and multiplication, and consequently, the unique absolute value.

The systems of Birkhoff, as well as Millman and Parker, which we discuss below, take this theorem for granted, that is, they adopt as an axiom the existence of a bijection between a geometric line and the real number line $(\R,<)$.

\subsection{Thales' theorem}

\pvn 1. In the context of Thales' theorem, the crucial proposition is  that parallel projection $f$ of segments lying on a line $l$ onto a line $p$ is a similarity map; see Fig. \ref{Projection}.
\[f:l\mapsto p.\]

\begin{theorem}\label{BSz1} (\cite{ref_BSz72}, pp. 158, 216) Parallel projection of one line onto another line is a similarity map. \end{theorem}
    
\begin{figure}[!ht]
\centerline{\includegraphics[scale=0.8]{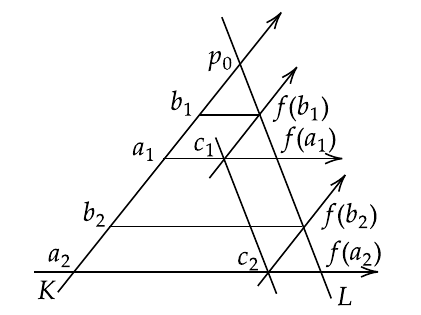}\includegraphics[scale=0.8]{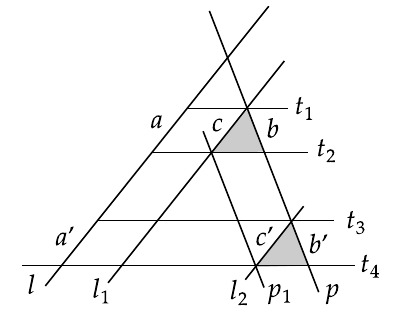}}
\caption[p]{Properties of parallel projections, \cite{ref_BSz72}, p. 215  (left), schematized proof (right). }\label{Projection}
\end{figure}
\begin{proof}
First, we show that the parallel projection defines a measure, that is, it satisfies the conditions \eqref{Measure}. We focus on the first condition, as it relates to a geometric insight that we will explore further in our paper.
The second condition consists in showing that the parallel projection is compatible with the relation $\textbf{B}$.

 Fig. \ref{Projection} (left) represents Borsuk and Szmielew's diagram, where 
$f$ maps points. In the neighboring diagram, we interpret 
$f$ as mapping line segments.

Let  
\[t_1\parallel t_2\parallel t_3 \parallel t_4, \ \ \ l\parallel l_1\parallel l_2,\ \ \ p\parallel p_1, \ \  f(a)=b,\ f(a')=b'.\] 

The aim is to show the relationship:
\[a\equiv a'\Rightarrow b\equiv b'.\]

 In Fig. \ref{Projection} (left), the quadrilateral $a_1b_1f(b_1)c_1$ is a parallelogram. Thus, $a\equiv c$. Similarly, we show that $a'\equiv c'$. Therefore, $c\equiv c'$.\footnote{The properties of parallelograms applied in this argument are stated in Euclid's Propositions I.33 and I.34.}

Triangles $\triangle c_1f(b_1)f(a_1)$ and $\triangle c_2f(b_2)f(a_2)$, or the grey triangles in the right diagram, are equiangular. By the Side-Angle-Side congruence criterion, we have, $f(b_1)f(a_1)\equiv f(b_2)f(a_2)$, or in the right diagram $b\equiv b'$.  

Second, the parallel projection $f$ satisfies conditions \eqref{Measure},  one can define a measure:
\[\varphi(a)=|f(a)|.\]

As observed earlier, given a value on one line segment,  a measure is determined, which means that $\varphi$
extends uniquely to the set of all line segments.

Now, since $\varphi$ is a measure, by Theorem \ref{M2}, there exists a real number $\lambda$ such that
\[\varphi(a)=|f(a)|=\lambda |a|.\]

This means that  $\varphi$ is a similarity map,  with $\lambda$ being the similarity scale. 
\end{proof}
\pvn 2. Thales' theorem is phrased in terms of similarity mappings and reduces to the statement that the parallel projection 
$f$ is a similarity map (\cite{ref_BSz72}, p. 216). 

The substance of this theorem is  as follows: For parallel projection $f$ of  $l$ onto $p$ there exists a real number $\lambda$ such that for any line segments $a, b$ on $l$ obtains (see Fig. \ref{TbyBSz}):
\[f(a)=\lambda a,\ \     f(b)=\lambda b.    \]

\begin{figure}[!ht]
\centerline{\includegraphics[scale=0.8]{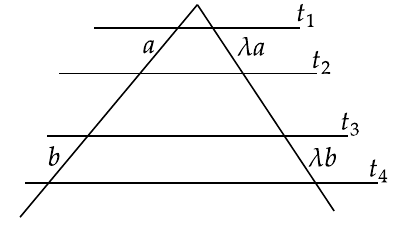}}\caption{Thales' theorem by Borsuk and Szmielew.}\label{TbyBSz}
\end{figure}

\subsection{Similar triangles}
\pvn 1. As a consequence of Thales' theorem, Borsuk and Szmielew obtain the following: (\cite{ref_BSz72}, p. 216--217):

In a triangle $\triangle ABC$, if lines $p_1, l_1$, parallel to $AB$ and $AC$ respectively, intersect the sides of  triangle $\triangle ABC$, cutting line segments $a$ on side $AC$,
$b$ on side $BC$, and $c$ on side $AB$, then the following equalities hold:
\[\frac a{a'}=\frac b{b'}=\frac c{c'},\]
{where} $AC=c'$, $BC=b'$, and $AB=c'$; see Fig. \ref{BSzVI4}.

Indeed, taking into account parallels $p, p_1$, by \ref{BSz1}, we obtain
\[b=\lambda a,\ \ b'=\lambda a'.\]

Similarly, considering the parallels  $l, l_1$, we obtain
\[b=\mu c,\ \ b'=\mu c'.\]

Finally, due to the arithmetic of real numbers, it follows that:
\[\frac a{a'}=\frac {\lambda a}{\lambda a'}=\frac b{b'}=\frac {\mu c}{\mu c'}  =\frac c{c'}.\] 

\begin{figure}[!ht]
\centerline{\includegraphics[scale=0.8]{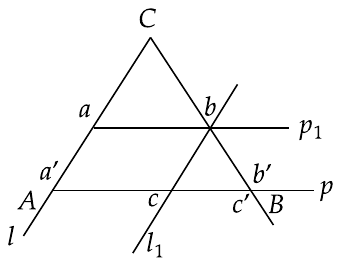}}\caption{Between Thales' theorem and Euclid's VI.4.}\label{BSzVI4}
\end{figure}

\pvn 2. At the top of this train of propositions, Borsuk and Szmielew place their interpretation of Euclid's Proposition VI.4 (\cite{ref_BSz72}, p. 217): 
\textit{In equiangular triangles, the sides about equal angles are proportional}; in this context, proportional means the equality of quotients within the arithmetic of real numbers.

Indeed, if the triangles $T$, with sides $a, b, c$, and $T'$, with sides $a', b', c'$,  are equiangular, then by cutting a copy of side $a$ on side $a'$ and drawing a parallel $l$ to side $p$, we construct a copy of the triangle $T$ inside triangle $T'$; see Fig. \ref{BVI4}. Then, by the previous result,
\[\frac a {a'}=\frac b{b'},\]
or, as required
\[\frac ab=\frac {a'}{b'}\]
The same holds for the pair $c, c'$.
\begin{figure}[!ht]
\centerline{\includegraphics[scale=0.9]{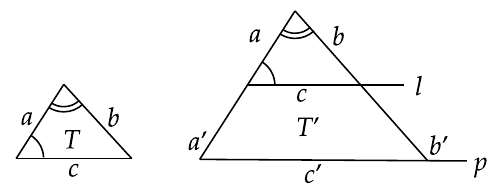}}\caption{Euclid' Proposition VI.4 by Borsuk and Szmielew.}\label{BVI4}
\end{figure}
\section{Birkhoff}

 \pvn 1. Birkhoff \cite{bir1932} designed a system that enables one to ''present the simplest geometric facts``. As usual in 20th-century systems, he adopts point and line as undefined terms. Furthermore, a real number  $d(A,B)$, called the distance between two points, is an undefined relation between two points. Similarly, a real number $\angle AOB\,(mod\,2\pi)$, called an angle formed by  points $A, O, B$, is an undefined relation between three points.  In fact, Birkhoff combines the concept of congruence with an interpretation in terms of real numbers.

The system is based on four axioms.

1. Postulate of Line Measure. A set of points $\{A, B, ...\}$ on any line can be put into one--to--one correspondence with the real numbers $\{a, b, ...\}$ so that $|b-a| = d(A,B)$ for all points $A$ and $B$.

This postulate enables the transportation of line segments.

2. Point-Line Postulate. There is one and only one line, $l$, that contains any two given distinct points $P$ and $Q$.

Non-intersecting lines are called parallel.

3. Postulate of Angle Measure. A set of rays $\{l, m, n...\}$ through any point $O$ can be put into one-to-one correspondence with the real numbers $a(mod 2\pi)$ so that if $A$ and $B$ are points (not equal to $O$) of $l$ and $m$, respectively, the difference $a_m - a_l (mod 2\pi)$ of the numbers associated with the lines $l$ and $m$ is $\angle AOB$.

 This postulate enables the transportation of angles. By the following convention, it also introduces the radian measure of angles: ''Two half lines $l, m$ through $O$ are said to form a \textit{straight angle} if $\angle lOm=\pi$``. (\cite{bir1932}, p. 332)  
 
 Interestingly, in the \textit{Elements}, this corresponds to Proposition I.13, where, instead of $\pi$, it states ''two right angles".

4. Postulate of Similarity. Given two triangles $ABC$ and $A'B'C'$ and some constant $k>0$, $d(A', B') = kd(A, B)$, 
$d(A', C')=kd(A, C)$ and \mbox{$\angle B'A'C'=\angle BAC$}, then \mbox{$d(B', C')=kd(B,C)$}, \mbox{$\angle C'B'A'= \angle CBA$}, and $\angle A'C'B' = \angle ACB$ 

This postulate implies Wallis's axiom, which, in standard systems, is proven to be equivalent to Euclid's Parallel Postulate or the Playfair Axiom.

By definition, similar triangles are equiangular, and their corresponding sides are proportional. 

Fig. \ref{figBirk}  represents Postulate IV schematically. It is Euclid's Proposition VI.6.
\begin{figure}[!ht] 
\centering
\includegraphics[scale=0.7]{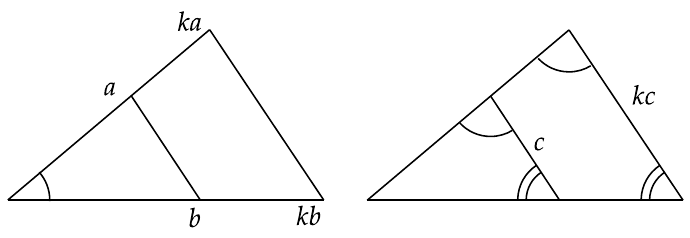}
\caption {Birkhoff's Postulate IV.} \label{figBirk}
\end{figure}

\pvn 2. Birkhoff shows that from his postulates follow Euclid' Proposition VI.4 and VI.5.
He also derives other propositions, such as I.5,  I.32, I.33--34, I.47, and Playfair's version of the parallel axiom.

 Ultimately, Birkhoff  seeks to prove Euclid's VI.33
 (ee Fig.\ref{birkhoff2} ):
\[\frac{\angle POR}{arc\,PR}=\frac{\angle POQ}{arc\,PQ}.\]

\begin{figure}[!ht] 
\centering
\includegraphics[scale=0.7]{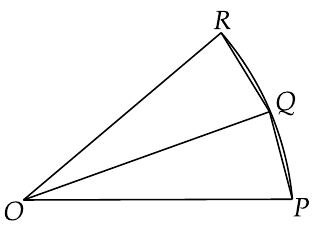}
\caption {Birkhoff's version of Euclid's VI.33.} \label{birkhoff2}
\end{figure}

To this end, he applies continuity arguments and a simplified proportion. Finally, he writes:
''Hence the (sensed) angle $\angle POQ$ coincides with the (sensed) arc lengthen $PQ$ subtended on the unit circle`` \cite[p. 345]{bir1932}.

 However, this argument relies on the following supposition: ''On the basis of the preceding theorem, Euclidean arc length can be defined in the usual manner and the angle $POQ$ ($P, Q$ on the circle) may be defined as the arc $PQ$ on the unit circle ($r=1$)``  \cite[p. 344]{bir1932}.

The usual definition is based on the Riemann integral, which assumes the radian measure. Clearly, Postulate III explicitly adopts the radian measure of angles.

\pvn 3.   Birkhoff does not prove Thales' theorem; therefore, we present our version based on his postulates.

Thales' theorem, specifically part VI.2b, follows from Postulate IV. 

Indeed, 
if $\frac{a}{c}=\frac{b}{d}$, than $\frac{d}{c}=\frac{b}{a}$. Taking $k=\frac{b}{a}$ and applying Postulate IV, we conclude that the respective trianlges are equiangular, which means the lines $l$ and $p$  are parallel; see Fig. \ref{figVI2b} (left).  

Now, for part VI.2a, suppose the lines $l$ and $p$ are parallel, but the required proportion does not hold,  i.e., $\frac{a}{c}\neq\frac{b}{d}$. By the arithmetic of real numbers, there exists $d'$ such that the equality holds: $\frac{a}{c}=\frac{b}{d'}$. Applying the previous part of the proof to this
proportion, we obtain that the lines $l$ and $q$ are parallel. However, this contradicts  Playfair's Axiom and the uniqueness of parallel to $l$ passing through the endpoint of line segment $b$; see Fig. \ref{figVI2b} (right). 

Note that the argument relies on a simple rule of the arithmetic of real numbers, which states that for three numbers $a, b, c$, there exists a fourth number $d'$
such that the following equality holds:
$\frac{a}{c}=\frac{b}{d'}$. Its ancient Greek counterpart is called the fourth proportional. Indeed, we include it as an axiom for the theory developed in Book V. However, it is employed implicitly, and Euclid's proof of VI.2 does not make use of it \cite[p. 47--51]{ref_BP}.

\section {Millman--Parker}

\pvn 1.  Millman-Parker's system \cite{F} is a mixture of Hilbert's synthetic approach and the
 metric space technique applied to the plane $G$; in effect, $G$ turns out to be $\R\times\R$. They take \textit{point}, \textit{line} and \textit{plane} as a primitive concepts, and  introduce incidence axioms:
 
(i) For every two points $A$, $B$ there is a line $l$ with $A\in l$ and $B\in l$.

(ii) Every line has at least two points.

Instead of congruence of line segments, they introduce a distance map. 

\begin{definition} \cite[p. 28]{F} A distance is a function  $d: G \times G \rightarrow \mathbb {R}$ such that for all $A, B \in G$:
\begin{enumerate}
\item  $d(A, B) \geq 0$;

\item $d(A, B) =0$ if and only if $A = B$;

\item $d(A, B) = d(B, A)$.
\end{enumerate}
\end{definition}

 Then, the congruence is interpreted through the distance map:
 \[AB\equiv CD \Leftrightarrow d(A,B)= d(C,D).\]

The relationship \textit{lying  between}, similarly, is defined by the distance map:
\[ A - B - C \Leftrightarrow  d(AB) + d(BC) = d(AC).\]

Borsuk and Szmielew established a bridge between synthetic and metric geometry: starting with Hilbert-style axioms, they introduced a metric and a coordinate system. Millman and Parker, on the other hand, simply introduce a coordinate system by definition.  

\begin{definition} \cite[p. 30]{F}
A function $f: l \rightarrow \mathbb R$ is a ruler (or coordinate system) for the line $l$ if:
\begin{enumerate}
\item $f$ is a bijection;

\item for each pair of points $A$ and $B$ on $l$  
$$|f(A) -f(B)|=d(A, B).$$
\end{enumerate}
\end{definition}

In other words, $f$ is an isometry, given that a metric on the line $(\R,<)$
 is introduced by the absolute value. 
 
 Let us reiterate, since none of the geometers pay attention to this fact: as the real numbers form a real-closed field, there exists the unique order compatible with addition and multiplication, and consequently, the unique absolute value on the real numbers.
\pvn 2. In Millman-Parker's system, the proof of Thales' theorem is based on the division of a line segment into equal parts.

\begin{theorem}\cite[p. 59]{F} \label{mill0} Any segment ${AB}$ can be divided into $ n $ equal parts.
\end{theorem}

The proof employs the bijection between the straight line and the real numbers. Let $A$ and $B$ lie on the line \( l \). Setting \( A \) as the origin of the coordinate system, we divide the length of the segment, the real number \( d = d(A,B) \), by \( n \), obtaining the real number \( \frac{d}{n} \). Let ${AA_1}$ be preimage of \( \frac{d}{n} \). By successively marking segments of length \( \frac{d}{n} \), we obtain points $A_1, A_2,..., A_n$ such that $d(A_i,A_{i+1})=\frac dn$. 

\hfill{$\Box$}

In \cite{F}, it is given as an exercise; however, from the perspective of the techniques employed, it is crucial to understand how, in a given system of geometry, one can divide a line segment. 
Note that in the \textit{Elements}, the counterpart of this theorem, namely VI.9, is based on Thales' theorem. More precisely, in VI.9, Euclid marks off equal segments on one arm of an angle and, using VI.2, proves that parallel lines through the endpoints of these segments determine equal parts on the other arm. This is the key argument in Millman-Parker's proof of  Thales' theorem.

The next theorem provides a method for dividing a segment into equal parts without referencing real numbers.

\begin{theorem} \cite[p. 231]{F}\label{mill1}
Let $l_1$, $l_2$, $l_3$ be distinct parallel lines. Let $t_1$ intersect $l_1$, $l_2$, $l_3$ at $A$, $B$, $C$, respectively,  and let $t_2$ intersect $l_1$, $l_2$, $l_3$ at $D$, $E$ i $F$, respectively. If
${AB}\equiv {BC}$, then ${DE}\equiv {EF}$ .
\end{theorem}

\begin{figure}[!ht] 
\centering
\includegraphics[scale=0.7]{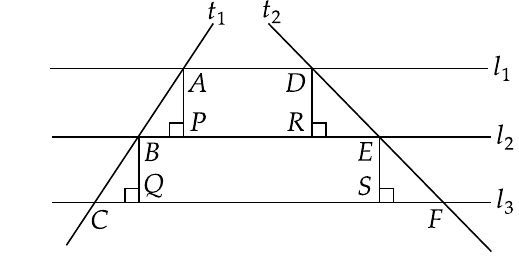}
\caption[p1] {Theorem \ref{mill1}, \cite{F}, p. 231.} \label{millman1}
\end{figure}

\begin{proof}

Let $P$ and $R$  be  feet of perpendiculars from $A$ and $D$ onto $l_2$, respectively, and $Q$ and $S$  be  feet of perpendiculars from $B$ and $E$ onto $l_3$, respectively; see Fig. \ref{millman1}. 
Since $AB\equiv BC$  and $\angle A=\angle B$, the right-angled triangles
$\triangle APB$ and $\triangle BQC$ are congruent.

The assumption on parallel lines implies
\[AP\equiv DR,\ \ BQ\equiv ES.\]

Since $AP\equiv BQ$, it follows that $DR\equiv BQ$.

Since $\angle D=\angle E$, the right-angled triangles
$\triangle DRE$ and $\triangle ESF$ are congruent, which implies that $DE\equiv EF$.
\end{proof}

From this point on, we apply the symbol $AB$ for the line segment or its length.  This convention simplifies notation, while the context always determines the correct meaning.

Here is Millman-Parker's proof of Thales' theorem. 
\begin{theorem} \cite[p. 231]{F} \label{millman4}
Let $l_1$, $l_2$ i $l_3$ be parallel lines. Let $t_1$ and $t_2$ be two transversals which intersect  $l_1$, $l_2$, $l_3$ at $A$, $B$, $C$ and $D$, $E$, $F$ with $A - B - C$ as shown in Fig. \ref{millman2}. Then 
$$\frac{BC}{AB} =\frac{EF}{DE}.$$
\end{theorem}

\begin{figure}[!ht] 
\centering
\includegraphics[scale=0.7]{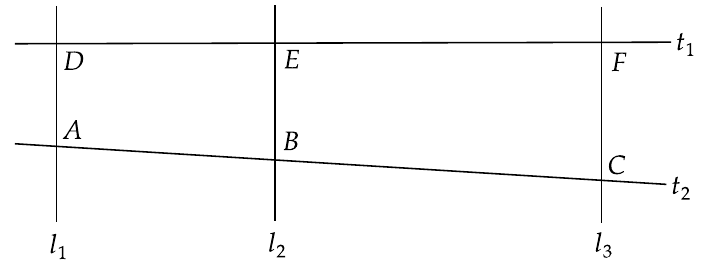}
\caption {Thales' theorem by Millman and Parker.} \label{millman2}
\end{figure}

\begin{proof} 
For the proof, it is shown that numbers $\frac{BC}{AB}$ and $\frac{EF}{DE}$ cannot differ, that is:
$$(\forall n \in \mathbb{N}) \left| \frac{BC}{AB}-\frac{EF}{DE} \right| <\frac{1}{n}.$$ 

This identity criterion is equivalent to the Archimedean axiom formulated in terms of an ordered $(\F,+,\cdot,0,1,<)$:
\[(\forall r\in \F_+)(\exists n\in\N)\,(\tfrac 1n<r). \]
or simply
\[\lim\limits_{n\rightarrow\infty}\frac 1n=0.\]

(1) For any $ n $, an integer $p$ is defined by:
\begin{equation}\label{mp1}p = \max\left\{k \in \mathbb{N} : k \leq \frac{nBC}{AB} \right\}.\end{equation}

From this definition, it follows that:
 \begin{equation}\frac{p}{n}\leq \frac{BC}{AB}< \frac{p+1}{n}, \label{0} \end{equation}
or, in an equivalent form that carries more geometrical significance:
\begin{equation} \label{1}   p\frac{AB}{n}\leq BC(p+1)\frac{AB}{n}. \end{equation}

Now, by Theorem \ref{mill0},  the segment ${AB}$ is divided into $n$ segments each of length $\frac{AB}{n}$, determining the points $A_1,  ..., A_{n-1}$.
 
Then, laying off $p + 1$ segments of length  $\frac{AB}{n}$ along ${BC}$,  determines points $B_1, B_2, ..., B_{p}$ on ${BC}$, with $B_{p+1}$ lying beyond $C$; see Fig. \ref{millman3}.

\begin{figure}[!ht] 
\centering
\includegraphics[scale=0.7]{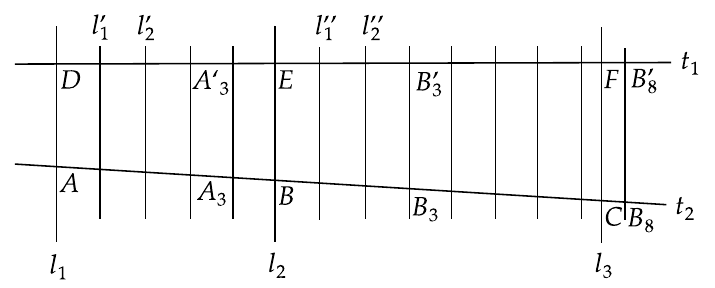}
\caption {Approximating ${BC}$ using segments of length $\frac{AB}{n}$ ($n=5$, $p=7$).} \label{millman3}
\end{figure}

(2) Let  $l_i'$ be the parallel to $l_1$ passing through $A_i$;  it meets $t_1$ at a point $A_i'$. 

Similarly, let  $l_j''$ be the parallel to $ l_1$ passing through $B_j$; it meets  $t_1$ at a point $B_j'$. 

By Theorem \ref{mill1}, the points $A_i'$  divide the $DE$ into equal segments each of length $\frac{DE}{n}$.

The segments ${B_j'B'_{j+1}}$ approximating the line segment ${EF}$, are all of the of length $\frac{DE}{n}$. Moreover,  inequalities analogous to  \ref{1} hold:
\begin{equation}p\frac{DE}{n}\leq {EF}<(p+1)\frac{DE}{n}.\label{2} \end{equation}
 
 (3) Due to the arithmetic of real numbers,  inequalities \ref{2} can be we  rewritten in the following form:
$$-\frac{p + 1}{n}< -\frac{EF}{DE}\leq -\frac{p}{n}.$$

Adding this to the inequality \ref{0} gives
$$-\frac{1}{n}<\frac{BC}{AB}-\frac{EF}{DE}<\frac{1}{n},$$
which implies that
$$\left| \frac{BC}{AB}-\frac{EF}{DE} \right| <\frac{1}{n}.$$
\end{proof}

Theorem \ref{millman4} relates to Euclid's version VI.2a. Version VI.2b could proceed in the same way as in Hilbert's or Birkhoff's system.

\section{From nonstandard approach to the hyperreal plane}
\subsection{Generalizing previous results}
1. The geometrical argument for Thales' theorem, as presented in the approaches of Borsuk and Szmielew, as well as Millman and Parker, relies on the observation that parallel lines cutting equal segments on one arm of an angle also cut equal segments on the other. We can summarize it as follows:

Let the arms $p, l$ of an angle be intersected by the parallel lines $t_1, t_2, t_3, t_4$, which cut segment $a, a'$ on $p$   and $b, b'$  on $l$.  When $a\equiv a'$, it follows that  $b\equiv b'$; see Fig. \ref{nsa1} (left): 
\[t_1\parallel t_2\parallel t_3\parallel t_4,\ a\equiv a'\Rightarrow b\equiv b'.\]
\begin{figure}[!ht] 
\centering
\includegraphics[scale=0.7]{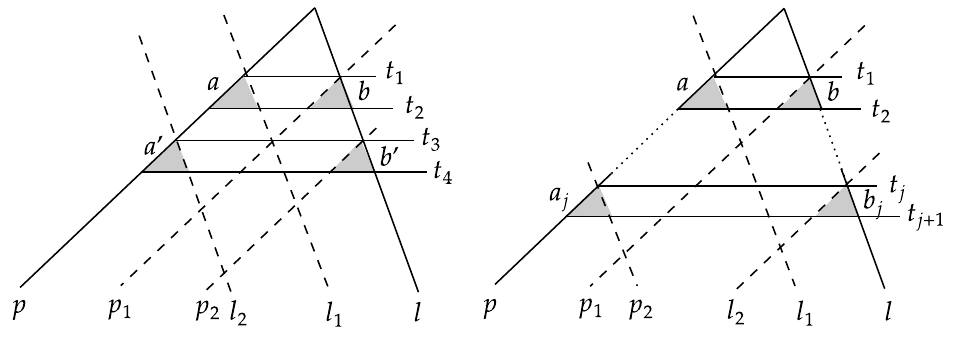}
\caption {Cutting equal segments on arms of an angle.} \label{nsa1}
\end{figure}

The proof is based on the properties of parallelogram stated in Euclid's Propositions I.33--34 \cite[p. 79]{blaszczyk2021}. Millman and Parker, instead of using parallelograms, apply rectangles; see Fig.  \ref{millman1}.

Note, however,  that this argument also applies to infinitely many parallel lines 
$\{t_j\}_{j\in J}$. That is, if parallel lines $t_j$, cutting arms $p, l$ of an angle in such way that segments $a_j$ on the arm $p$ are equal, then the corresponding segments $b_j$ on the arm $l$ are also equal; see Fig. \ref{nsa1} (right).

\pvn 2. With this construction, we can prove Euclid's VI.9 without referencing Thales' theorem. Suppose we want to divide the line segment  $CB$ into $m$ equal parts; see Fig. \ref{nsa2} (left). To achieve this, lay off $m$ segments  of length $\varepsilon$ along ray $AC$, and enumerate them as $(\varepsilon_j)_{j\leq m}$. Join $A$, the endpoint of $\varepsilon_m$, to $B$. By drawing parallels to $AB$ through the endpoints of $\varepsilon_j$, we obtain equal parts on segment $CB$.

\pvn 3. We can also reverse this argument. In Fig. \ref{nsa2} (left), $\varepsilon$ represents equal segments on side $AC$, 
while $\delta$ represents equal segments on side  $BC$. Let us enumerate these segments as $(\varepsilon_j)_{j\leq m}$ and   $(\delta_j)_{j\leq m}$, with $m\in\N$, and let $t_1$ join the endpoints of $\varepsilon_1$ and $\delta_1$. 

\begin{figure}[!ht] 
\centering
\includegraphics[scale=0.7]{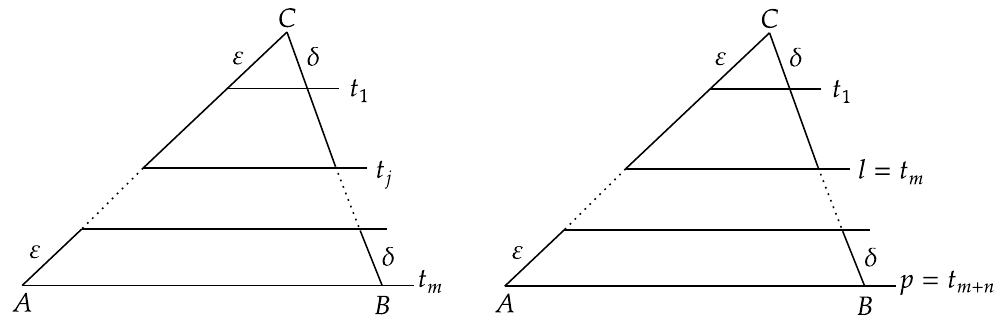}
\caption {Reversing Euclid's VI.9 (left). Co-measurable case of VI.2 (right).} \label{nsa2}
\end{figure}
 Lines  $(t_j)_{j\leq m}$ parallel to $t_1$ passing through endpoints of  $\varepsilon_j$ determine  equal segments on side $BC$. Since the first of these segments equals $\delta$, it follows that all segments on $BC$ are equal to $\delta$. Moreover, since $AB$ 
 runs through the endpoints of $\varepsilon_m$, $\delta_m$, the side $AB$ and the line $t_m$ coincide.

\subsection{Thales' theorem on the real plane $\R\times \R$}
 In the plane $\R\times\R$, a segment $AC$ has length $d$, and we assume there exist points 
$A_1, ..., A_n$ such that segments $AA_1, ..., A_iA_{i+1},..., A_nC$ are equal, each having length $\frac dn$.   

It what follows, we assume  $a, b, c, d, \varepsilon, \delta$ 
denote both segments and their lengths, initially as real numbers.

\pvn 1. Let us first  consider a commensurable case of  Thales' theorem:  

\begin{theorem} In the triangle  $\triangle ACB$,
    if parallel lines $l, p$ cut line segments $a, b$ and $c, d$ on the sides  $AC$ and $BC$ of the triangle,  respectively, and there exists a segment $\varepsilon$ such that
$\frac ab= \frac{m\varepsilon} {n\varepsilon}$, with $m,n\in\N$, then it follows that $\frac ab=\frac cd$:
\[l\parallel p,\ \ \frac ab= \frac{m\varepsilon} {n\varepsilon} \Rightarrow \frac ab=\frac cd,\ \ \ \mbox{given}\ \ \ a, b, c, d, \varepsilon\in\R_+.\]
\end{theorem}

\begin{proof}
 Marking off $m+n$ segments $\varepsilon$ on side $AC$, we draw $m+n$ lines through the endpoints of these segments, parallel to $l$. The first parallel determines a segment $\delta$, and $BC$ is divided into $m+n$ segments, each of length $\delta$.

Since $c=m\delta$ and $d=n\delta$, it follows that
\[\frac ab= \frac{m\varepsilon} {n\varepsilon}=\frac{m\delta} {n\delta}= \frac cd.\]   
\end{proof} 
Fig. \ref{nsa2} (right) illustrates this proof.

Although the proof is simple, in the next section, we will generalize it by taking hyperintegers $K, L$, instead of integers $m, n$, that is, special infinite numbers.

\pvn 2. The above observations enable us to apply techniques from nonstandard analysis to prove Thales' theorem while remaining within the real plane 
$\R\times\R$. To this end, we briefly sketch the basics of hyperreal numbers in this section. The crucial concept is that of hyperfinite integers, which enable us to reconstruct arguments \ref{mp1}--\ref{2}, as well as the argument developed in \S\,9.1. 

The ordered field of hyperreals, or nonstandard real numbers, \linebreak {$(\mathbb R^*,+,\cdot,0,1,<)$} is the extension of real numbers, where $\mathbb R^*=\mathbb R^\mathbb N/\mathcal U$, with $\mathcal U$ being a non-principial ultrafilter on $\mathbb N$ \cite{blasczyk2021Galileo,goldblatt1998}. 
Thus, a hyperreal number is represented by an equivalence class determined by a sequence of real numbers, 
 \[[(r_1,r_2,...)]\in\mathbb R^*.\]
 
A constant sequence $[(r,r,..)]$ represents the standard real number $r$.
 
 The absolute value is defined in the same way as in any ordered field.
 
We define  the class of infinitely small, infinitely large, and limited numbers  as follows (where $n$ ranges over $\mathbb N$):
 \[x\in\Omega\Leftrightarrow (\forall n)(|x|<\tfrac 1n), 
 \ \   x\in \Psi\Leftrightarrow (\forall n)(|x|>n),\ \  x\in \mathbb L\Leftrightarrow (\exists n)(|x|<n).\]

One can easily verify the following relationships:
\[\Omega+\Omega,\ \Omega\cdot \Omega\subset \Omega,\ \ \ \Omega\cdot \mathbb L\subset \Omega,\ \ \ \mbox{and}\ \ x\in\Omega\Leftrightarrow x^{-1}\in \Psi,\ \ x\neq 0.\]

Due to these definitions, any positive real number 
$r$ is greater than any infinitesimal hyperreal number.


The set $\Psi$ includes the set of hyperintegers $\mathbb N^*$, which extends the set of natural numbers. The structure  $(\mathbb N^*, +, \cdot,0,1)$ forms a nonstandard (uncountable) model of Peano arithmetic. The elements of $\mathbb N^*$, denoted below as $K, L$,  are represented by equivalence classes of sequences of natural numbers, such as $[(n_1,n_2,...)]$. 

\pvn 3. Returning to geometry, let the angle $\angle ACB$ be placed in the real plane $\R\times\R$. Let $ l$ and $p$ intersect its arms, forming the line segments  $a, b$ on $AC$ and $c, d$ on $BC$, respectively. We assume that $a, b, c, d\in \R$, and we use these symbols to denote both the segments and their lengths.  Our goal is to prove the following implication (see Fig. \ref{fignsa4.6}):
\[l\parallel p\Rightarrow \frac ab=\frac cd.\]

\begin{figure}[!ht] 
\centering
\includegraphics[scale=0.7]{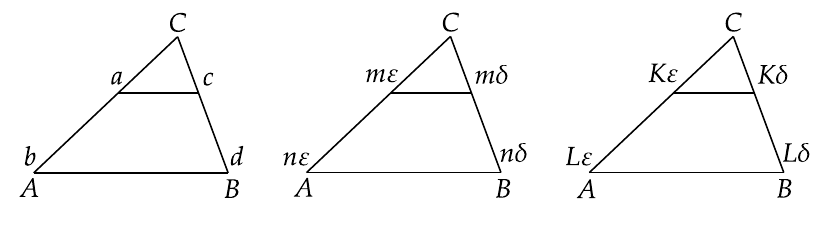}
\caption {Thales' theorem from co-measurable case (middle) to nonstandard approach (left).} \label{fignsa4.6}
\end{figure}

We consider two cases: (a) co-measurable and (b) non-comeasurable.

(Ad a) By co-measurable line segments in nonstandard sense, we mean that there exists an infinitesimal $\varepsilon$ such that 
$$a=K\varepsilon,\ \ \ b=L\varepsilon,$$
where  $K, L\in\N^*$, and $\varepsilon\in\Omega$. 

First, on the arm $AC$, we set  $K+L$ segments $\varepsilon$. Let us enumerate them as in the section \S\,9.1(2) above $\{\varepsilon_j: j\leq K+L\}$.

Let $t_1$ be  parallel to $l$, passing through the endpoint of the first one,  $\varepsilon_1$. It determins the triangle with vertex $C$ and sides of lengths $\varepsilon$ and $\delta$; see Fig. \ref{fignsa4} (left). 
By the previous considerations, these parallels cut $K+L$ equal segments on $BC$, each with length $\delta$.

As in the section \S\,9.2(1), due to the arithmetic of an ordered field, we obtain
\begin{equation}\frac ab= \frac{K\varepsilon} {L\varepsilon}=\frac{K\delta} {L\delta}= \frac cd.\label{nsa3.5}\end{equation}

(Ad b) Non-commessurable case. Suppose $a$ and $b$ are not co-measurable. Let $K$ be any hyperinteger, and set
\[\varepsilon= \frac aK.\]

Since a standard real number $a$ is a limited hyperreal, $a\in \mathbb L$, the number $\varepsilon$ is infinitesimal, $\varepsilon\in\Omega$, and we can express $a$ as $a=K\varepsilon$.

Then,  for some hyperinteger $L$, the following inequalities hold:\footnote{
By an additional argument, we could show that $L\varepsilon< b$; however, this is not necessary.}
\begin{equation}\label{nsa3} L\varepsilon\leq b< (L+1)\varepsilon.\end{equation}

Similarly to \ref{mp1}, within the hyperintegers,  such an $L$ exists.  

Indeed, given $K=[(k_1,k_2,...)]$ and $a$ is represented by the equivalence class $[(a,a,...)]$, we have 
\[\varepsilon=[(\tfrac a{k_1}, \tfrac a{k_2}, ...)]. \]

For each $j$, there exists an integer $l_j$,  such that
\[l_j\frac a{k_j}\leq b <(l_j+1)\frac a{k_j}, \ \ \  l_j\in\N.\]

Setting $L=[(l_1,l_2,...)]$, we obtain \ref{nsa3}.

\hfill{$\Box$}
\begin{figure}[!ht] 
\centering
\includegraphics[scale=0.7]{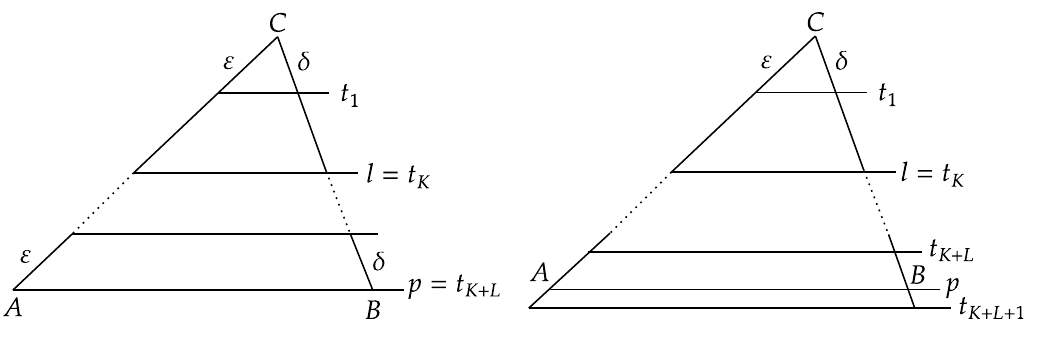}
\caption {Thales' theorem in nonstandard approach.} \label{fignsa4}
\end{figure}

Now, suppose the equality $\frac ab=\frac cd$ does not hold. Then the equality  $\frac ba=\frac dc$ also does not hold. Suppose  $\frac ba$ is greater:
\begin{equation}\label{sigma1}\frac ba-\frac dc=\sigma, \ \ \ \sigma\in\R_+.\end{equation}

On the ray $AC$, we set  $K+L+1$ segments $\varepsilon$. Let us enumerate them as  $\{\varepsilon_j: j\leq K+L+1\}$.

Let $t_1$ be  parallel to $l$, passing through the endpoint  of $\varepsilon_1$. It determins the line segment  $\delta$ on side $BC$.

Lines $t_j$ parallel to $t_1$, are drawn through the endpoints of the segments $\varepsilon_j$, determining the line segments $\delta_j$ on arm $BC$, each equal to $\delta$; see Fig. \ref{fignsa4}.  

It follows that:
\[(K+L)\varepsilon < AC < (K+L+1)\varepsilon,    \]
accordingly, 
\[(K+L)\delta < BC< (K+L+1)\delta. \]

Moreover,
\begin{equation}a=K\varepsilon ,\ \ L\varepsilon < b<(L+1)\varepsilon,\ \ \ c=K\delta,\ \ L\delta< d< (L+1)\delta.\label{nsa4}
\end{equation}

From \ref{sigma1} and \ref{nsa4}, it follows:
\begin{equation}\sigma=\frac ba-\frac dc< \frac {L+1}K-\frac LK=\frac 1K.\label{nsa4.5}\end{equation}

Since $\frac 1K$ is infinitesimal, it contradicts the assumption that $\sigma$ is a positive real number. 

\hfill{$\Box$}

\subsection{Thales' theorem on the hyperreal plane $\R^*\times \R^*$}

\pvn 1. The arguments developed in the previous subsection translate to the hyperreal plane  $\R^*\times \R^*$. 
 
 Suppose $a, b, c , d$ are positive hyperreal numbers, $a, b, c , d\in\R^*_+$.

To prove Thales' theorem (VI.2a), we consider two cases: (a) line segments $a, b$ are co-measurable (in the nonstandard sense), (b)  line segments $a, b$ are  non-comeasurable.

(Ad a)  The co-measurable case proceeds analogously to the previous one and concludes with formula \ref{nsa3.5}.

(Ad b) Suppose the equality $\frac ba=\frac dc$ does not hold and $\frac ba$ is greater:
\begin{equation}\label{sigma}\frac ba-\frac dc=\sigma, \ \ \ \sigma\in\R^*_+.\end{equation}

Whether $\sigma$ is infinitesimal or limited, we can find an infinitely large number $K$ such that the inequality holds:
\begin{equation}\label{nsa5}\frac 1 K< \sigma. \end{equation}

Indeed, given 
\[  \sigma=[(r_1,r_2,..)],\]
for each index $j$, due to the Archimedean axiom, there exists an integer $k_j$ such that the inequality holds
\[\frac 1{r_j}<k_j.\]
 Setting
\[K=[(k_1,k_2,...)],\]
we obtain
\[\frac 1\sigma<K,\]
or \ref{nsa5}.
Then, with this $K$, we reiterate arguments \ref{nsa3} to \ref{nsa4.5}.

Specifically, given
\[a=[(a_1, a_2,...)],\ \  b=[(b_1, b_2,...)],\ \ \varepsilon=[(\tfrac{a_1}{k_1}, \tfrac{a_2}{k_2},...)],\]
we find $l_j$ such that
\[l_j\tfrac{a_j}{k_j}\leq b_j<(l_j+1)\tfrac{a_j}{k_j},\]
and set
\[L=[(l_1,l_2,...)].\]

As a result, the contradiction follows:
\[\sigma=\frac ba-\frac dc<  \frac 1K<\sigma.\]
\hfill{$\Box$}
\pvn 2.  Arguments developed above in \S\,9 rest, in part, on arithmetic, in part, on geometry. Let us summarize them.

Suppose $l\parallel p$ and $\frac ab \neq\frac cd$.  Let
\[0<\sigma=\frac ba-\frac dc.\]

There exists an integer $k$ such that 
\[\frac ba-\frac dc<\frac 1k<\sigma.\]

This leads to a contradiction: 
\[\sigma<\frac 1k< \sigma.\]

Let us note, that we can formalize this as  a sentence:
\begin{equation}\tag{T}(\forall a, b, c, d\in \R)(\exists k\in\N)(\tfrac ba-\tfrac dc<\tfrac 1k< \tfrac ba-\tfrac dc).\end{equation}

 Statement T is true in Euclidean geometry over real plane $\R\times \R$.

By the transfer principle (\cite{goldblatt1998}, ch. 4) it is also true
in Euclidean geometry over hyperreal plane $\R\times \R$:
\begin{equation}\tag{T*}(\forall a, b, c, d\in \R^*)(\exists k\in\N^*)
(\tfrac{b}{a}-\tfrac{d}{c}<\tfrac{1}{k}< \tfrac{b}{a}-\tfrac{d}{c}).
\end{equation}

As for the arguments justifying T,  they can also be formalized in a similar manner -- for example, using Tarski-style geometry -- in a form  transferable from geometry over $\R\times \R$ to geometry over $\R^*\times \R^*$. 

\section{Area Method}

In reconstructing Euclid's theory of similar triangles, the 20th-century systems, instead of Thales' theorem, adopt or prove some other propositions from Book VI.

Hilbert base his theory of proportion on Proposition VI.4, while its proof explored the reverse of Euclid's III.21. Birkhoff adopts VI.6 as an axiom.

Borsuk and Szmielew, as well as Millman and Parker, based their approach on VI.9 and the arithmetic of real numbers. Borsuk and Szmielew, in particular, explore the density of dyadic numbers in 
$(\R,<)$, while Millman and Parker apply another version of the Archimedean axiom, 

In Section 9, we generalize these arguments and show that they can be developed in the non-Archimedean field of hyperreal numbers.

All these systems interpret Euclidean proportion as a equality of divisions.
Moreover, whether in the arithmetic of line segments or real numbers, they explore the concept of the fourth proportional. Although the idea of the fourth proportional arises from Greek mathematics, Euclid did not refer to it in Book VI.

In this final section, we present an approach that adopts VI.1 as an axiom. Then, the proof of Proposition VI.9 aligns with Euclid's original proof.

\subsection{Area method}

The area method, pioneered in \cite{C},  is a~technique of proving theorems and constructing solutions in Euclidean geometry.  \cite{E}  provides its axiomatic description.  In \cite{ref_BP}, we presented a~model for these axioms.

From the perspective of  formal systems, the language of the area method includes one kind of variables, and symbols of  a~binary, $\overline{{\phantom{ab}}}$, and a~ternary function, $S$. We also need  the language of a commutative  field characteristic $0$, that is,  symbols of binary functions, $+,\ \cdot$ (sum and product), and unary functions $-, ^{-1}$ (an opposite and inverse element),  as well as constants $0,1$, and finitely many constants and $r$. 

Less formally, there are three primitive notions in the area method: point,  length of a~directed segment, 
and a~signed area of a~triangle. An ordered pair of points is called a~directed segment, an ordered triple -- a~triangle. In what follows, capital letters $A$, $B$, $C$, etc., stand for points.
The length of a~ directed segment, $\overline{AB}$, in short, is an element  of an ordered field. Similarly, the signed area of a~triangle,   $S_{ABC}$,  in short, is an element of the ordered field. $\overline{AB}$ and  $S_{ABC}$  can be positive, negative, or zero and they are  processed in the arithmetic of a commutative field. 

To model Euclidean geometry, we need some definitions that we apply in axioms.

\begin{definition}
Points $A, B, C$ are collinear iff $S_{ABC}=0$.
\end{definition}

\begin{definition}

Two segments   $AD$ and $BC$, where $A\neq D$ and $B \neq C$, are   parallel,  iff $S_{ABC}=S_{DBC}$. For this relation, we adopt the standard symbol $AD \parallel BC$.

\end{definition}

\begin{definition}
For three points $A$, $B$ and $C$, the Pythagorean difference, denoted by $P_{ABC}$, is defined by
$$P_{ABC}=\overline{AB}^2+\overline{BC}^2-\overline{AC}^2.$$

\end{definition}

\begin{definition}
Two segments   $DB$ and $CA$, where $D\neq B$ and $C \neq A$, are   perpendicular
 iff $P_{DCA}=P_{BCA}$. This relation is denoted by $DB \perp CA$.

\end{definition}

Here are the axioms for the area method \cite{E}.

 A1. $ \overline{A B}=0$ if and only if $ A$ and $B$ are identical.

A2. $S_{ABC} = S_{CAB}$.

A3. $S_{ABC} = -S_{BAC}$.

A4. If $S_{ABC} = 0$, then $\overline{AB} +\overline{BC} = \overline{AC}$ (Chasles' axiom).

A5. There are points $A$, $B$ and $C$ such that $S_{ABC}\ne 0$ (not all points are collinear).

A6. $ S_{ABC} = S_{DBC} +S_{ADC} +S_{ABD}$ (all points are in the same plane).\footnote{The idea of signed area  originates from Hilbert's \textit{Foundations of Geometry}.  In \cite[ch. 5]{ref_DH72}  he proves the theorem that is a~counterpart of axiom A6.}

A7. For each element $ r$ of $F$, there exists a~point $P$, such that $S_{ABP} = 0$ and $\overline {AP} = r \overline{AB}$ (construction of a~point on a~line).

A8. If $A\ne B$, $S_{ABP} = 0,\overline {AP} = r\overline{AB}$, $S_{ABP'} = 0$ and $\overline{AP'} = r\overline {AB}$, then $P = P'$.

A9. If $PQ \parallel CD$ and $ \frac{\overline {PQ}}{\overline{CD}}=1$, then $DQ \parallel PC$ (Euclid's proposition I.33).

A10. If $S_{PAC} \ne 0$ and $S_{ABC} = 0$, then $\frac{\overline{AB}}{\overline{AC}}=\frac{S_{PAB}}{S_{PAC}}$ (Euclid's proposition VI.1).

A11. If $C\ne D$ and $AB \perp CD$ and $EF \perp CD$, then $AB \parallel EF$.

A12. If $A\ne B$, $AB \perp CD$ and $AB \parallel EF$, then $ EF \perp CD$.

A13. If $FA \perp BC$ and $S_{FBC} = 0$, then $4 \cdot S_{ABC}^2 =\overline {AF} ^2 \overline {BC} ^2 $ (formula for the area of a~triangle).

The schemes of Euclid's proof and the area method proof are almost identical. The only difference is that within the area method, one must respect the order of the endpoints of line segments and the vertices of triangles, whereas Euclid arbitrarily permutes the names of a triangle's vertices in his proofs.

Below, we present a proof of Thales' theorem within the area method (see Fig. \ref{fig7}) \cite{ref_BP}.

(1) From the assumption $DE \parallel BC$, by definition, we obtain the equality of signed areas $S_{DEB}=S_{DEC}$.

(2) By the arithmetic of the filed:
$$\frac{S_{DEB}}{S_{DAE}}=\frac{S_{DEC}}{S_{DAE}}.$$

 (3) By A2, we can permute the names of vertices. Then   by A10, the following equalities hold:
$$\frac{S_{BDE}}{S_{DAE}}=\frac{\overline{BD}}{\overline{DA}}, \ \ \ \frac{S_{CDE}}{S_{DAE}}=\frac{\overline{CE}}{\overline{EA}}.$$

(4) By transitivity of equality, we obtain,

$$\frac{\overline{BD}}{\overline{DA}}=\frac{\overline{CE}}{\overline{EA}}.$$

\hfill{$\Box$}
\subsection{GCLC prover}

The Area Method enables the mechanization of Euclid's propositions. If understanding arguments in synthetic geometry involves grasping the axioms and rules of inference, then in automated proofs, it involves eliminating points. Elimination lemmas specify this procedure. Given that, an automated proof proceeds as follows:
\begin{enumerate}
\item 	The thesis of a theorem is translated into an expression in the Area Method language. 
\item	Given some starting points, new points are introduced, one by one, through the allowed constructions (construction stage). 
\item	Each point introduced in the construction stage is eliminated based on elimination lemmas, but in reverse order, i.e., the last constructed is the first in the elimination process, etc. (elimination stage). 
\item	The process reaches identity $1=1$ or $0=0$ and stops.
\end{enumerate}

Since the point elimination method is suitable for algorithmization, it is used to create programs for automatic proving theorems, the so-called provers. An example of such a prover is GCLC \cite{GCLC}. It is a tool for visualizing and creating mathematical drawings, and automatically proving geometric theorems. Prover GCLC generates traditional proofs based on the geometric properties of objects: the coordinates of the entered points are not taken into account in the automatic proof. Point, area, and segment are primitive concepts; they do not have any numerical expressions in the automatic proof; they are symbols. In this sense, it automatically produces synthetic proofs for geometric theorems, which justifies its use in proving theorems of Euclid’s geometry. 

Thales' theorem, as well as most of the propositions from Book VI of Euclid's Elements, have been successfully proven using GCLC. Below is an example of such a proof, illustrating how the Area Method can be applied effectively within the system to automate these classical geometric arguments.

 \begin{figure}[!h]
 \centering
\includegraphics[scale=0.65]{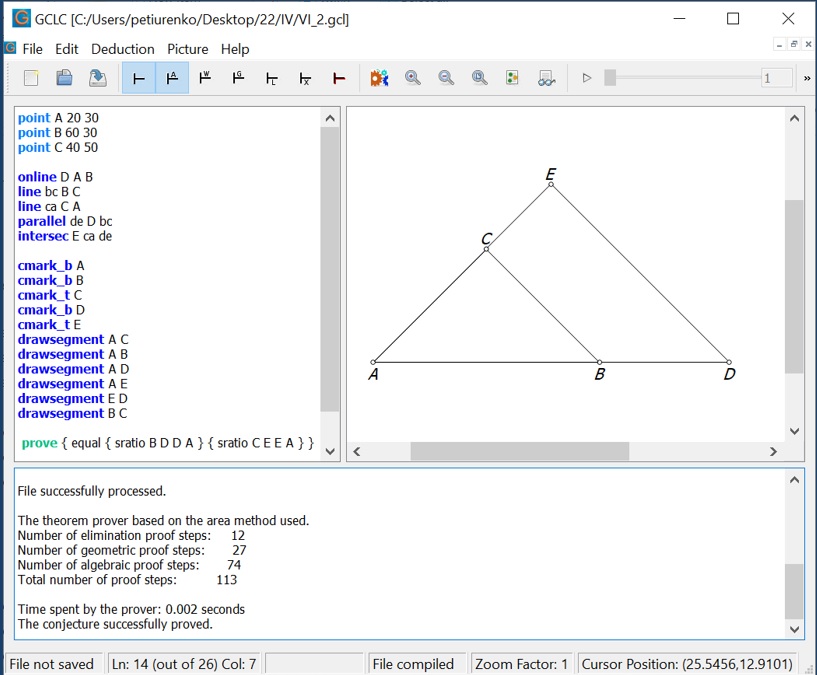}
\caption{Theorem VI.2, construction at GCLC} \label{th2}
\end{figure}

\begin{proof}
Case 1. Let $DE$ be drawn parallel to one of the sides $BC$ of triangle $ABC$. We need to show that as $BD$ is to $DA$, so $CE$ is to $EA$.

Automatic proof will be performed only in the case when only allowed constructions are used in the geometric interpretation of the theorem:
\begin{itemize}
\item point – defines any point;
\item	line – defines a straight line that passes through two points;
\item	intersect – defines the point of intersection of two lines;
\item	online – defines a point on the line;
\item	midpoint – defines a point as the middle of the segment;
\item	parallel – defines a line parallel to the given line passing through the given point;
\item	towards – division of the line segment in the given ratio;
\item	translate – translation a line segment;
\item	foot – defines a line perpendicular to a given line passing through a given point outside the line;
\item	perp – defines a line perpendicular to the given line passing through a given point on the line.
\end{itemize}

 Construction steps:
\begin{lstlisting}[breaklines]	
	point A 20 30 
	point B 60 30      
	point C 40 50  
	online D A B
	line bc B C
	line ca C A
	parallel de D bc
	intersec E ca de
\end{lstlisting}

The coordinates of the entered points are not taken into account in the automatic proof. Point, area, and segment are primitive concepts; they do not have any numerical expressions in the automatic proof; they are symbols.

Theorem thesis for case 2 in terms of automatic proof: 

\begin{lstlisting}[breaklines]
prove { equal { sratio B D D A } { sratio C E E A } }}
\end{lstlisting}

In ordinary mathematical language, it can be written like this:
 $$ \frac{\overline{BD}}{\overline{DA}}= \frac{\overline{CE}}{\overline{EA}}.$$
 
Case 2. Let the sides $AB$ and $AC$ of triangle $ABC$ be cut proportionally such that as $BD$ is to $DA$, so $CE$ is to $EA$. Let $DE$ have been joined. Prove that $DE$ is parallel to $BC$. 

This case is a~bit more complicated. We have the following construction steps:
\begin{lstlisting}[breaklines]	
	point A 20 30 
	point B 60 30      
	point C 40 50  
	towards D A B 0.3
	towards E A C 0.3
\end{lstlisting}

Theorem thesis for case 2 in terms of automatic proof:

 \begin{lstlisting}[breaklines] 
 prove { parallel D E B C } 
 \end{lstlisting}

In ordinary mathematical language, it means $DE\parallel BC$.

In the towards command we introduce a~concrete 0.3 parameter which means that $\frac{AD}{DB} = \frac{AE}{EC} = \frac{1}{3} $. This may mean that the proof we get is not general. Let's analyze the automatic proof GCLC:
$$S_{DBC}{=}{ S_{EBC}}$$
$$S_{DBC}{=}{ S_{BCE}}$$
$$S_{DBC}{=}{ ( S_{BCA} +  ( 0.3  \cdot  ( S_{BCC} +  ( -1  \cdot  S_{BCA}))))}$$
$$S_{BCD}{=}{ ( S_{BCA} +  ( 0.3  \cdot  ( 0  +  ( -1  \cdot  S_{BCA}))))}$$
$$S_{BCD}{=}{ ( 0.7  \cdot  S_{BCA})}$$
$$( S_{BCA} +  ( 0.3  \cdot  ( S_{BCB} +  ( -1  \cdot  S_{BCA})))){=}{ ( 0.7  \cdot  S_{BCA})}$$
$$( S_{BCA} +  ( 0.3  \cdot  ( 0  +  ( -1  \cdot  S_{BCA})))){=}{ ( 0.7  \cdot  S_{BCA})}$$
$$0 =0$$

If we change 0.3 to $ r $ and 0.7 to $ 1-r $, where $ r \in \mathbb {R} $, the proof will not change, and hence proof is general.

\end{proof}

Note: If we look at the first line of the proof, we can see that the prover reformulated the parallelity to the equality of two areas of triangles (because this is exactly how the parallelity of two lines is defined in the area method). Hence, the thesis of theorem VI. 1 case 2 can be formulated as 
\begin{lstlisting}[breaklines]
proof {equal {signed_area3 DBC} {signed_area3 EBC}}.
\end{lstlisting}

In ordinary mathematical language, it means $S_{DBC}=S_{EBC}$.

 The proof will remain the same if we use the above thesis.

\section{Conclusions}
Let us go back to our schematic reconstruction of Euclid's proof of Thales theorem discussed in \S\,2.4:
\[l\parallel p \xrightarrow[I.37]{} T_1=T_2 \xrightarrow[V.7]{} \frac{T_1}{T}=\frac{T_2}{T} \xrightarrow[VI.1]{}\frac{b}{a}=\frac{d}{c},\]
and 
\[\frac{b}{a}=\frac{d}{c} \xrightarrow[VI.1]{} \frac{T_1}{T}=\frac{T_2}{T}\xrightarrow[V.9]{} T_1=T_2\xrightarrow[I.39]{} l\parallel p. \]

In a more schematic form, it is a relationship between parallelism and proportion of line segments:
\[l\parallel p \xleftrightarrow[]{}\frac{b}{a}=\frac{d}{c},\]
or in a Euclidean stylization:
\[l\parallel p \xleftrightarrow[]{} {b}:{a}::{d}:{c}.\]

 20th-century systems reinterpreted the concept of proportion 
\mbox{${b}:{a}::{d}:{c}$}  in terms of the arithmetic of line segments or real numbers. Due to this change, they omit reference to Proposition VI.1 and the mixed proportion involved in that proposition -- namely, the proportion between triangles and line segments, as presented in Section \S\,2.1. They also omit a proportion between triangles, as shown in the above scheme.

Taking VI.1 as an axiom, the Area Method, on one hand, renews the Euclidean technique of proportion, and on the other hand, makes it mechanical, bringing it into 21st-century mathematics. In a way, it realizes the dream of Ian Mueller, a great admirer of Greek mathematics, as he stated:

\pvn ``Since VI.1 is the only important use of Eudoxus' definition in book VI
[Definition 5, Book V], it is clear that any theory enabling one to prove VI.1 and standard laws of proportion would suffice as a basis for book VI" \cite[p. 156]{IM}.

\bibliographystyle{agsm}
\bibliography{biblio}

\end{document}